\documentclass[article,nojss,nofooter]{jss}
\usepackage{orcidlink,thumbpdf,lmodern}

\usepackage{amsmath,amsthm,amssymb}
\usepackage{algorithm}
\usepackage{algorithmicx}
\usepackage{algpseudocode}
\usepackage{bookmark}
\usepackage{booktabs}
\usepackage{paralist}
\usepackage[section]{placeins}

\newtheorem{theorem}{Theorem}
\newtheorem{lemma}[theorem]{Lemma} 
\newtheorem{corollary}{Corollary}

\providecommand{\abs}[1]{\lvert#1\rvert}
\providecommand{\sgn}{\ensuremath{\mathrm{sgn}} }

\makeatletter
\newcommand\blfootnote[1]{%
  \begingroup
  \renewcommand{\@makefntext}[1]{\noindent#1}
  \renewcommand\thefootnote{}\footnote{#1}%
  \addtocounter{footnote}{-1}%
  \endgroup
}
\makeatother

\author{Janez Komelj~\orcidlink{0009-0002-3474-2654}\\ Vnanje Gorice, Slovenia\\ {\footnotesize E-mail: \email{jkomelj@siol.net}}}
\Plainauthor{Janez Komelj}

\title{The Bivariate Normal Integral via Owen's $T$ Function as a Modified Euler's Arctangent Series}
\Plaintitle{The Bivariate Normal Integral via Owen's T Function as a Modified Euler's Arctangent Series}
\Shorttitle{Owen's $T$ Function as a Modified Euler's Arctangent Series}

\Abstract{
The Owen's $T$ function is presented in four new ways, one of them as a series similar to the Euler's arctangent series divided by $2\pi$, which is its majorant series. All possibilities enable numerically stable and fast convergent computation of the bivariate normal integral with simple recursion. When tested $\Phi_\varrho^2(x,y)$ computation on a random sample of one million parameter triplets with uniformly distributed components and using double precision arithmetic, the maximum absolute error was $3.45\cdot 10^{-16}$. In additional testing, focusing on cases with correlation coefficients close to one in absolute value, when the computation may be very sensitive to small rounding errors, the accuracy was retained. In rare potentially critical cases, a simple adjustment to the computation procedure was performed -- one potentially critical computation was replaced with two equivalent non-critical ones. All new series are suitable for vector and high-precision computation, assuming they are supplemented with appropriate efficient and accurate computation of the arctangent and standard normal cumulative distribution functions. They are implemented by the \proglang{R} package \pkg{Phi2rho}, available on CRAN. Its functions allow vector arguments and are ready to work with the \pkg{Rmpfr} package, which enables the use of arbitrary precision instead of double precision numbers. A special test with up to $1024$-bit precision computation is also presented. 
}

\Keywords{Owen's $T$ function, bivariate normal integral, Euler's arctangent series, recursion, \proglang{R} package \pkg{Phi2rho}}
\Plainkeywords{Owen's T function, bivariate normal integral, Euler's arctangent series, recursion, R package Phi2rho}

\Address{
  Janez Komelj\\
  Vnanje Gorice\\
  Slovenia\\
  E-mail: \email{jkomelj@siol.net}
}

\hypersetup{
  pdfinfo={
    Authors={Janez Komelj (Vnanje Gorice, Slovenia)},
    Abstract={The Owen's T function is presented in four new ways, one of them as a series similar to the Euler's arctangent series divided by $2\pi$, which is its majorant series. All possibilities enable numerically stable and fast convergent computation of the bivariate normal integral with simple recursion. When tested $\Phi_\varrho^2(x,y)$ computation on a random sample of one million parameter triplets with uniformly distributed components and using double precision arithmetic, the maximum absolute error was $3.45\cdot 10^{-16}$. In additional testing, focusing on cases with correlation coefficients close to one in absolute value, when the computation may be very sensitive to small rounding errors, the accuracy was retained. In rare potentially critical cases, a simple adjustment to the computation procedure was performed -- one potentially critical computation was replaced with two equivalent non-critical ones. All new series are suitable for vector and high-precision computation, assuming they are supplemented with appropriate efficient and accurate computation of the arctangent and standard normal cumulative distribution functions. They are implemented by the R package Phi2rho, available on CRAN. Its functions allow vector arguments and are ready to work with the Rmpfr package, which enables the use of arbitrary precision instead of double precision numbers. A special test with up to 1024-bit precision computation is also presented.},
    Comments={24 pages with 1 figure. In this version, the equation (40c) was corrected. In Recursions 1 and 2, all indices in the last two steps were fixed (decremented by 1) and the corresponding explanatory text was adjusted. Phi2rho and results unchanged.},
    Category={math.NA},
    MSC-Class={41A63, 41A58, 62H10},
    Journal-ref={American Journal of Computational Mathematics, 2023, Vol. 13, No. 4, 476-504}, 
    doi={http://dx.doi.org/10.4236/ajcm.2023.134026}
  }
}

\begin{document}


\blfootnote{\fbox{
\begin{minipage}{0.97\textwidth}
This is an electronic reprint of the original article published in American Journal of Computational Mathematics, 2023, Vol. 13, No. 4, 476--504, \doi{10.4236/ajcm.2023.134026}. This reprint differs from the original in reference style, pagination, and typographic details. In this version, the equation \eqref{Owen6c} was corrected ($\arctan r$ changed to $\arctan\frac{1}{r}$). In Recursions 1 and 2, all indices in the last two steps were fixed (decremented by 1) and the corresponding explanatory text was adjusted. \pkg{Phi2rho} and results unchanged.
\end{minipage}
}}

\section{Introduction}

Let $\Phi(x)=\frac{1}{\sqrt{2\pi}}\int_{-\infty}^x e^{-\frac{1}{2}t^2}\,\mathrm{d}t$ be the standard normal cumulative distribution function (cdf) and $\varphi(x)=\frac{1}{\sqrt{2\pi}}\;e^{-\frac{1}{2}x^2}$ the standard normal probability density function (pdf). The bivariate standard normal cdf
\begin{equation}\label{Phirxy1}
   \Phi_\varrho^2(x,y)=\tfrac{1}{2\pi\,\sqrt{1-\varrho^2}}\;\int_{-\infty}^x \int_{-\infty}^y e^{-\frac{s^2-2\varrho st+t^2}{2\,(1-\varrho^2)}}\,\mathrm{d}s\;\mathrm{d}t
\end{equation}
with a correlation coefficient $\varrho\in(-1,1)$ has the bivariate standard normal pdf
\begin{equation}\label{phirxy}
   \varphi_\varrho^2(x,y)=\tfrac{\partial^2\Phi_\varrho^2(x,y)}{\partial x\,\partial y}= \tfrac{1}{2\pi\,\sqrt{1-\varrho^2}}\;e^{-\frac{x^2-2\varrho xy+y^2}{2\,(1-\varrho^2)}}.
\end{equation}
If $\varrho=0$, then $\Phi_0^2(x,y)=\Phi(x)\,\Phi(y)$, while for $\abs{\varrho}=1$ we have the degenerated cases
\begin{equation}\label{Phi1xy}
   \Phi_{\pm 1}^2(x,y)=
   \begin{cases}
      \min\{\Phi(x),\Phi(y)\},     & \mbox{if }\varrho=+1,\\
      \max\{\Phi(x)+\Phi(y)-1,0\}, & \mbox{if }\varrho=-1.
   \end{cases}
\end{equation}
Additionally, for $\varrho\in(-1,1)$ define
\begin{equation*}
   L(x,y,\varrho)=\tfrac{1}{2\pi\,\sqrt{1-\varrho^2}}\,\int_x^{\infty}\int_y^{\infty} e^{-\frac{s^2-2\varrho st+t^2}{2\,(1-\varrho^2)}}\,\mathrm{d}s\;\mathrm{d}t,
\end{equation*}
$L(x,y,-1)=\max\{1-\Phi(x)-\Phi(y),0\}$ and $L(x,y,1)=1-\max\{\Phi(x),\Phi(y)\}$. The functions $\Phi_\varrho^2(x,y)$ and $L(x,y,\varrho)$ are symmetric regarding $x$ and $y$, and because they are complementary, it suffices to analyze only one of them. In this respect, $L$ is often preferred, but not in this paper where some equations for $L$ are also rewritten using $\Phi_\varrho^2(x,y)=L(-x,-y,\varrho)$ and some other auxiliary equations.

Noting that \eqref{Phi1xy} implies $\Phi_1^2(x,x)=\Phi(x)=2\,\Phi_0^2(x,0)$ and $\Phi_{-1}^2(x,-x)=0=\Phi_0^2(x,0)+\Phi_0^2(-x,0)-\frac{1}{2}$, the equations from \citep[p.~937, 26.3.19 and 26.3.20]{Abramowitz1972}, which do not cover these cases if $x\ne 0$, as we get an indefinite expression $\frac{0}{0}$ for the auxiliary correlation coefficients $\varrho_x$ and $\varrho_y$, can be transformed into
\begin{equation}\label{Phirxy2}
   \Phi_\varrho^2(x,y)=
   \begin{cases}
      \tfrac{1}{4}+\tfrac{1}{2\pi}\arcsin\varrho,            & \mbox{if }x=y=0,\\
      \Phi_{\varrho_x}^2(x,0)+\Phi_{\varrho_y}^2(y,0)-\beta, & \mbox{if }x\ne 0\mbox{ or }y\ne 0,
   \end{cases}
\end{equation}
where $\beta=0$ if $xy>0$ or $xy=0$ and $x+y\ge 0$, and $\beta=\frac{1}{2}$ otherwise,
\begin{equation}\label{Phirxy3}
   \varrho_x=\left\{
   \begin{array}{lcr}
      0 & \mbox{if }\abs{\varrho}=1\mbox{ and }x\ne 0\mbox{ and }x-y\,\sgn(\varrho)=0 & 0\\
      \tfrac{(\varrho x-y)\,\sgn(x)}{\sqrt{x^2-2\varrho xy+y^2}} & \mbox{otherwise} &
      \tfrac{(\varrho y-x)\,\sgn(y)}{\sqrt{x^2-2\varrho xy+y^2}}
   \end{array}
   \right\}=\varrho_y,
\end{equation}
where $\sgn(x)=1$, if $x\ge 0$ and $\sgn(x)=-1$, if $x<0$. Note that although $\abs{\varrho}$ is not close to $1$, $\abs{\varrho_x}$ or $\abs{\varrho_y}$ can be, and that $\abs{\varrho}=1$ and the lower equations \eqref{Phirxy3} imply $\abs{\varrho_x}=1$ and $\abs{\varrho_y}=1$.

Since \citep[p.~936, 26.3.9 and 26.3.8]{Abramowitz1972} imply
\begin{equation}\label{Phirx0}
   \Phi_\varrho^2(-x,0)=\Phi_\varrho^2(x,0)-\Phi(x)+\tfrac{1}{2}\quad\text{and}\quad
   \Phi_{-\varrho}^2(x,0)=\Phi(x)-\Phi_\varrho^2(x,0),
\end{equation}
and $\Phi_{-\varrho}^2(-x,0)=\frac{1}{2}-\Phi_\varrho^2(x,0)$, the problem how to compute $\Phi_\varrho^2(x,y)$ is reduced to the computation of $\Phi_\varrho^2(x,0)$ for $x\ge 0$ and $\varrho\in[0,1)$.

Extensive older bibliography exists about how $\Phi_\varrho^2(x,y)$ can be computed \citep[see][]{Gupta1963a}. Many old methods, e.g.\ \citeauthor{Donnelly1973}'s (\citeyear{Donnelly1973}) method based on \citeauthor{Owen1956}'s (\citeyear{Owen1956}) results, \citeauthor{Divgi1979}'s (\citeyear{Divgi1979}) method and the method of \citet{Drezner1990}, are sufficiently effective and accurate for the majority of present needs, but new methods still appear, e.g.\ \citep{Patefield2000,Genz2004,Meyer2013b,Fayed2014b}. More details and comparisons of some methods can be found in \citep{Terza1991,Gai2001,Agca2003}.

The Owen's $T$ function is defined by \citep[pp.~1078--1079]{Owen1956}
\begin{equation}\label{Owen1}
   \begin{split}
      T(h,a)&=\tfrac{1}{2\pi}\int_0^a\tfrac{e^{-\frac{1}{2}h^2(1+x^2)}}{1+x^2}\, \mathrm{d}x\\
      &=\tfrac{\arctan a}{2\pi}-\tfrac{1}{2\pi}\sum_{k=0}^\infty\tfrac{(-1)^k a^{2k+1}}{2k+1}\left(1-e^{-\frac{1}{2}h^2}\sum_{i=0}^k\tfrac{h^{2i}}{2^i i!}\right)
   \end{split}
\end{equation}
for $a,h\in\mathbb{R}$. It appears frequently in the extensive Owen's table of integrals involving Gaussian functions \citep[pp.~393--413]{Owen1980} and can be used for computing $\Phi_\varrho^2(x,y)$, as will be explained in the next section. The series \eqref{Owen1} is used in \citeauthor{Donnelly1973}'s (\citeyear{Donnelly1973}) algorithm and, based on it, is also realized in \code{pbinorm()} function from the \pkg{VGAM} package \citep{Yee2023}. Since it alternates, there are problems with numerical stability, e.g.\ using double precision arithmetic, it may provide only single precision results. More about Owen's $T$ function and older ways of calculating it are given by \citet{Patefield2000}, who developed a hybrid method to compute it. To achieve double precision accuracy, they use the series \eqref{Owen1} or one of five other methods. Which method is chosen and which its parameters are used depends on which of the $18$ regions $(h,a)$ belongs to \citep[see][p.~7, Fig.~2]{Patefield2000}. Their method is also implemented in \code{OwenT()} function from the \pkg{OwenQ} package \citep{Laurent2023}, hereafter \code{OwenQ::OwenT()}.

The method presented in this paper extends the theoretical knowledge and can be easily implemented. Its great advantage over the already mentioned methods, which enable $\Phi_\varrho^2(x,y)$ computation via Owen's $T$ function, is its simplicity, faster convergence and ensured numerical stability. Regarding other methods that have proven to be good or even excellent in terms of $\Phi_\varrho^2(x,y)$ computation speed and accuracy, e.g.\ the one implemented in the \code{pmvnorm()} function from the \pkg{mvtnorm} package \citep{Genz2009}, the main advantages are simplicity and independence from various constants. Both enable easy portability between different environments and using high-precision arithmetic. There is no need to adjust the number and precision of various constants, such as weights and knots in quadrature formulas, which are used in many $\Phi_\varrho^2(x,y)$ computational methods.

In Section \ref{sec:Background}, the background results are presented, which are the starting point for the development of a new method, as well as for the acceleration of the tetrachoric series, which was used for the benchmark values computation. In Section \ref{sec:Auxiliary}, auxiliary results are derived. In core Section \ref{sec:Main}, new series for $\Phi_\varrho^2(x,0)$ are derived in Theorems \ref{Theorem2} and \ref{Theorem3}, and in Corollary \ref{Corollary2} they are combined in a unified series for $\Phi_\varrho^2(x,y)$. However, in all cases, the series must be supplemented by the standard normal cdf and some of them also by the arctangent function computation. Four series for computing the Owen's $T$ function are presented in Theorem \ref{Theorem4}. Among them \eqref{Owen6b} can be viewed as a modified Euler's arctangent series. In Sections \ref{sec:Recursion} and \ref{sec:Results}, algorithmic and numeric issues are discussed, and testing results are presented.

Using double precision arithmetic and testing on a random sample of one million triplets $(x,y,\varrho)$ with uniformly distributed components, the maximum absolute error was $3.45\cdot 10^{-16}$, which is approximately $3.11$-times the machine epsilon $2^{-53}$. On the same random sample, but with $\varrho$ transformed so that more than half of $\abs{\varrho}$ were greater than $0.9999$, and by equivalent computing of two function values instead of $\Phi_\varrho^2(x,y)$ in rare cases with $\varphi_\varrho^2(x,y)>1$, the maximum absolute error was even slightly smaller.

All new series are suitable for vector and high-precision computation. The fastest one asymptotically converges as $\sum_k\frac{q^{k+1}e^{-q}}{(k+1)!}\,\frac{\min\{a^{2k},1\}}{\sqrt{k+1}\;(1+a^2)^k}$, where $q=\frac{1}{2}(1+a^2)h^2$. Using the \pkg{Rmpfr} package \citep{Maechler2023} and $1024$-bit precision computation, $199$ iterations were needed to compute $\Phi_{\sqrt{2}/2}^2(2.1,0)$, achieving an absolute error smaller than $2^{-1024}\approx 5.56\cdot 10^{-309}$.

\section{Background results}\label{sec:Background}

The bivariate standard normal cdf can be expanded into tetrachoric series \citetext{\citealp[p.~196]{Kendall1941}, \citealp[p.~940, 26.3.29]{Abramowitz1972}}
\begin{equation}\label{TetrachA}
   \Phi_\varrho^2(x,y)=\Phi(x)\,\Phi(y)+\varphi(x)\,\varphi(y)
   \sum_{k=0}^\infty H\!e_k(x)\,H\!e_k(y)\,\tfrac{\varrho^{k+1}}{(k+1)!},
\end{equation}
where $H\!e_k(x)$ are the ``probabilist's'' Hermite polynomials defined by
\begin{equation*}
   H\!e_k(x)=(-1)^k\,e^{\frac{1}{2}x^2}\,\tfrac{\mathrm{d}^k}
   {\mathrm{d}x^k}\;e^{-\frac{1}{2}x^2}\quad(k\in\mathbb{N}_0).
\end{equation*}
The tetrachoric series \eqref{TetrachA} converges a bit faster than a geometric series with a quotient $\varrho$ and slowly converges for $\abs{\varrho}$ close to or equal to $1$.

Let $\mathbb{N}_{-1}=\mathbb{N}\cup\{-1,0\}$. Extending the usual double factorial function definition for $k\in\mathbb{N}_0$ to $k\in\mathbb{N}_{-1}$ with $(-1)!!=1$, we have
\begin{equation*}
   H\!e_k(0)=
   \begin{cases}
      0,                       & \mbox{if $k$ is odd},\\
      (-1)^\frac{k}{2}(k-1)!!, & \mbox{if $k$ is even}.
   \end{cases}
\end{equation*}
Setting $y=0$ in \eqref{TetrachA}, transforming the summation index $\frac{k}{2}\mapsto k$, and using $(2k-1)!!=\frac{(2k)!}{2^k k!}$ for $k\in\mathbb{N}_0$, we get a slightly faster converging series
\begin{equation}\label{TetrachB}
   \Phi_\varrho^2(x,0)=\tfrac{1}{2}\,\Phi(x)+\tfrac{\varphi(x)}{\sqrt{2\pi}}
   \sum_{k=0}^\infty H\!e_{2k}(x)\,\tfrac{(-1)^k\varrho^{2k+1}}{2^k k!\,(2k+1)}.
\end{equation}
\citet[p.~7]{Vasicek1998a} found an alternative series for the computation of $\Phi_\varrho^2(x,0)$, which is a good alternative to the series \eqref{TetrachB} if $\varrho^2>\frac{1}{2}$.

Integrating the well known equation \citep[p.~352]{Plackett1954}
\begin{equation}\label{Phirxy4}
   \tfrac{\partial}{\partial\varrho}\,\Phi_\varrho^2(x,y)=
   \tfrac{\partial^2\Phi_\varrho^2(x,y)}{\partial x\,\partial y}=\varphi_\varrho^2(x,y)
\end{equation}
with respect to $\varrho$, using equations $\Phi_0^2(x,y)=\Phi(x)\,\Phi(y)$ and \eqref{Phi1xy}, we get
\begin{equation}\label{Phirxy5}
   \begin{split}
      \Phi_\varrho^2(x,y)&=\Phi(x)\,\Phi(y)+\int_0^\varrho\varphi_s^2(x,y)\, \mathrm{d}s\\
      &=\min\{\Phi(x),\Phi(y)\}-\int_\varrho^1\varphi_s^2(x,y)\,\mathrm{d}s\\
      &=\max\{\Phi(x)+\Phi(y)-1,0\}+\int_{-1}^\varrho\varphi_s^2(x,y)\,\mathrm{d}s.
   \end{split}
\end{equation}
Due to historical reasons, the variates $h$ and $k$ are often used instead of $x$ and $y$. We will use $h$ instead of $x$ but will not need $k$ in the role of $y$.

In the sequel, let $\varrho\in(-1,1)$ and $r\in\mathbb{R}$ be permanently connected by the one-to-one correspondence $r=\frac{\varrho}{\sqrt{1-\varrho^2}}$ and $\varrho=\frac{r}{\sqrt{1+r^2}}$. Notice that \eqref{Phirxy2} and $\arcsin\varrho=\arctan r$ imply $\Phi_\varrho^2(0,0)=\frac{1}{4}+\frac{1}{2\pi}\arctan r$. Setting $x=h$, $y=0$, and writing $e^{-\frac{1}{2}h^2}\int_a^b e^{\frac{1}{2}h^2}\varphi_s^2(h,0)\,\mathrm{d}s$ instead of $\int_a^b\varphi_s^2(h,0)\,\mathrm{d}s$, the equations \eqref{Phirxy5} can be rewritten to
\begin{equation}\label{Phirh0a}
   \begin{split}
      \Phi_\varrho^2(h,0)
      &=\tfrac{1}{2}\,\Phi(h)+\varphi(h)\,T_r(-\tfrac{1}{2}h^2)\\
      &=\min\{\Phi(h),\tfrac{1}{2}\}-\varphi(h)\,\widetilde{T}_r(-\tfrac{1}{2}h^2)\\
      &=\max\{\Phi(h)-\tfrac{1}{2},0\}+\varphi(h)\,\overline{T}_r(-\tfrac{1}{2}h^2),
   \end{split}
\end{equation}
where, using a substitution $s=\frac{x}{\sqrt{1+x^2}}$,
\begin{equation*}
   T_r(w)=\tfrac{1}{\sqrt{2\pi}}\int_0^\varrho\tfrac{1}{\sqrt{1-s^2}}\;
   e^{\frac{ws^2}{1-s^2}}\,\mathrm{d}s=\tfrac{1}{\sqrt{2\pi}} \int_0^r\tfrac{e^{wx^2}}{1+x^2}\,\mathrm{d}x\quad(w\in\mathbb{R}).
\end{equation*}
Analogously, $\widetilde{T}_r(w)$ and $\overline{T}_r(w)$ are defined for $w\le 0$ by integrals with limits $(\varrho,1)$ and $(-1,\varrho)$, and $(r,\infty)$ and $(-\infty,r)$ respectively, which do not exist for $w>0$. The computation of $\Phi_\varrho^2(h,0)$ is reduced to the computation of one of the integrals $T_r(-\frac{1}{2}h^2)$, $\widetilde{T}_r(-\frac{1}{2}h^2)$ and $\overline{T}_r(-\frac{1}{2}h^2)$.

The special cases \eqref{Phirh0a} of the equations \eqref{Phirxy5} in combination with \eqref{Phirxy2} are a starting point for some computation methods for $\Phi_\varrho^2(x,y)$ which are based on the integration. Owen's is one of them -- note that $\varphi(h)\,T_a(-\frac{1}{2}h^2)=T(h,a)$. Using $r$ instead of $a$ in the sequel and resolving the problem with $\frac{0}{0}$ as before, Owen's equation \citep[p.~416, 3.1]{Owen1980} can be rewritten as
\begin{equation}\label{Owen2}
   \Phi_\varrho^2(x,y)=
   \begin{cases}
      \tfrac{1}{4}+\tfrac{1}{2\pi}\arcsin\varrho,            & \mbox{if }x=y=0,\\
      \tfrac{1}{2}(\Phi(x)+\Phi(y))-T(x,r_x)-T(y,r_y)-\beta, & \mbox{if }x\ne 0\mbox{ or }y\ne 0,
   \end{cases}
\end{equation}
where $\beta=0$ if $xy>0$ or $xy=0$ and $x+y\ge 0$, and $\beta=\frac{1}{2}$ otherwise, and
\begin{equation}\label{Owen3}
   r_x=\left\{
   \begin{array}{lcr}
      0 & \mbox{if }\abs{\varrho}=1\mbox{ and }x\ne 0\mbox{ and }x-y\,\sgn(\varrho)=0 & 0\\
      \tfrac{y-\varrho x}{x\,\sqrt{1-\varrho^2}} & \mbox{otherwise} &
      \frac{x-\varrho y}{y\,\sqrt{1-\varrho^2}}
   \end{array}
   \right\}=r_y.
\end{equation}
Another Owen's important equation is \citep[p.~414, 2.7]{Owen1980}
\begin{equation}\label{Owen4}
   T(h,r)+T(rh,\tfrac{1}{r})=\tfrac{1}{2}(\Phi(h)+\Phi(rh))-\Phi(h)\,\Phi(rh)-\beta,
\end{equation}
where $\beta=0$ if $r\ge 0$ and $\beta=\frac{1}{2}$ if $r<0$. It implies $\varphi(h)\,T_1(-\frac{1}{2}h^2)=T(h,1)=\frac{1}{2}\Phi(h)\,(1-\Phi(h))$.
If $\varrho=\frac{\sqrt{2}}{2}$, then $r=1$. Using \eqref{Phirh0a} and the right equation \eqref{Phirx0}, we get
\begin{equation}\label{Phirh0b}
   \Phi_{\sqrt{2}/2}^2(h,0)=\Phi(h)(1-\tfrac{1}{2}\,\Phi(h))\quad\text{and}\quad \Phi_{-\sqrt{2}/2}^2(h,0)=\tfrac{1}{2}(\Phi(h))^2.
\end{equation}
Let $\varrho\in(-1,1)$, $\overline{\varrho}=\sgn(\varrho)\sqrt{1-\varrho^2}$,
$r=\frac{\varrho}{\sqrt{1-\varrho^2}}$, $\overline{r}=\frac{\overline{\varrho}}{\sqrt{1-\overline{\varrho}^2}}=\frac{1}{r}$ and $h\in\mathbb{R}$. By adding $\frac{1}{2}(\Phi(h)+\Phi(rh))$ on both sides of \eqref{Owen4} and using $\varphi(h)\,T_r(-\frac{1}{2}h^2)=T(h,r)$ and \eqref{Phirh0a}, it follows
\begin{equation}\label{Owen5}
   \Phi_\varrho^2(h,0)+\,\Phi_{\overline{\varrho}}^2(rh,0)=\Phi(h)+\Phi(rh)- \Phi(h)\,\Phi(rh)-\beta,
\end{equation}
where $\beta=0$ if $r\ge 0$ and $\beta=\frac{1}{2}$ if $r<0$.

If $\abs{\varrho}>\frac{\sqrt{2}}{2}$, then $\abs{r}>1$, $\abs{\overline{r}}<1$ and $\abs{\overline{\varrho}}<\frac{\sqrt{2}}{2}$. $\Phi_\varrho^2(h,0)$ can be computed by computing $\Phi_{\overline{\varrho}}^2(rh,0)$ and using \eqref{Owen5}. Consequently, the right equation \eqref{Phirx0} and the transformation of the parameters $rh\mapsto\overline{h}$ and $\frac{1}{r}\mapsto\overline{r}$, if needed, reduce the condition $\varrho\in(-1,1)$ to $\varrho\in[0,\frac{\sqrt{2}}{2}]$ with a corresponding $r\in[0,1]$. The problem how to compute $\Phi_\varrho^2(x,y)$ is reduced to the computation of $\Phi_\varrho^2(h,0)$ for $h\ge 0$ and $\varrho\in[0,\frac{\sqrt{2}}{2}]$.

The theorems in the sequel do not depend on the restriction to $\varrho^2\le\frac{1}{2}$ with a corresponding $\abs{r}\le 1$ found in \citep{Patefield2000}. However, if $\abs{r}>1$, the transformation of $rh\mapsto\overline{h}$ and $\frac{1}{r}\mapsto\overline{r}$ speeds up convergence significantly. This is, for example, also the case for the Vasicek's series and the tetrachoric series \eqref{TetrachB}, which transforms into
\begin{equation}\label{TetrachC}
   \Phi_\varrho^2(h,0)=\Phi(h)+\tfrac{1}{2}\Phi(rh)-\Phi(h)\Phi(rh)
   -\tfrac{\varphi(rh)}{\sqrt{2\pi}}\sum_{k=0}^\infty H\!e_{2k}(rh)\, \tfrac{(-1)^k\overline{\varrho}^{2k+1}}{2^k k!\,(2k+1)}-\beta,
\end{equation}
where $\overline{\varrho}=\sgn(\varrho)\sqrt{1-\varrho^2}$, $\beta=0$ if $r\ge 0$ and $\beta=\frac{1}{2}$ if $r<0$. Since the results calculated with the Vasicek's series in Section \ref{sec:Results} only supplement benchmark values calculated with the tetrachoric series, details about it are not given here.

\section{Auxiliary results}\label{sec:Auxiliary}

This paper pays attention to the estimation of the method error due to truncation of infinite series, while the analysis of rounding errors is out of its scope. The following lemma is intended for this. In it and all the theorems below, the phrase ``truncated after the $n^{th}$ term'' also means omitting the entire series. In these cases, set $n=-1$ and assume the empty sum is $0$.

\begin{lemma}
If the Euler's series
\begin{equation}\label{Euler1}
   \arctan r=\tfrac{r}{1+r^2}\sum_{k=0}^\infty\tfrac{(2k)!!}{(2k+1)!!} \left(\tfrac{r^2}{1+r^2}\right)^k\quad(r\in\mathbb{R})
\end{equation}
is truncated after the $n^{th}$ term, then the remainder
\begin{equation*}
   R_n=\tfrac{r}{1+r^2}\sum_{k=n+1}^\infty\tfrac{(2k)!!}{(2k+1)!!} \left(\tfrac{r^2}{1+r^2}\right)^k \quad(n\in\mathbb{N}_{-1})
\end{equation*}
has the same sign as $r$ and is bounded in absolute value by
\begin{equation}\label{Euler2}
   B_n\le\left\lvert R_n\right\rvert\le(1+r^2)\,B_n,\quad\text{where}\quad B_n=\tfrac{(2n+2)!!}{(2n+3)!!}\,\tfrac{\abs{r}}{1+r^2} \left(\tfrac{r^2}{1+r^2}\right)^{n+1}.
\end{equation}
\end{lemma}

\begin{proof}
Let $\alpha\in(0,1]$. Noting that $k\ge n+1\ge 0$ and $\int_0^1(1-x^2)^k\,\mathrm{d}x= \int_0^\frac{\pi}{2}\cos^{2k+1}t\,\mathrm{d}t=\frac{(2k)!!}{(2k+1)!!}$ for $k\in\mathbb{N}_0$ \citep[p.~397, 3.621, 4.]{Gradshteyn2015}, and using Tonelli's theorem, it follows
\begin{equation}\label{Euler3}
   \begin{split}
      S(r,\alpha)&=\sum_{k=n+1}^\infty\tfrac{(2k)!!}{(2k+1)!!} \left(\tfrac{\alpha\,r^2}{1+r^2}\right)^k=
      \sum_{k=n+1}^\infty\left(\tfrac{\alpha\,r^2}{1+r^2}\right)^k \int_0^1(1-x^2)^k\,\mathrm{d}x\\
      &=\int_0^1\sum_{k=n+1}^\infty\left(\tfrac{\alpha\,r^2\,(1-x^2)}{1+r^2} \right)^k\mathrm{d}x
      =(1+r^2)\left(\tfrac{\alpha\,r^2}{1+r^2}\right)^{n+1} \int_0^1\tfrac{(1-x^2)^{n+1}\,\mathrm{d}x} {1+(1-\alpha)\,r^2+\alpha\,r^2x^2}.
   \end{split}
\end{equation}
The denominator in the last integrand is an increasing function of $x$. Inserting $1$ and $0$ instead of $x$ implies
\begin{equation}\label{Euler4}
   \tfrac{(2n+2)!!}{(2n+3)!!}\,\tfrac{\alpha\,\abs{r}}{1+r^2} \left(\tfrac{\alpha\,r^2}{1+r^2}\right)^{n+1}\le \left\lvert\tfrac{\alpha\,r}{1+r^2}\,S(r,\alpha)\right\rvert\le \tfrac{(2n+2)!!}{(2n+3)!!}\,\tfrac{\alpha\,\abs{r}}{1+(1-\alpha)\,r^2} \left(\tfrac{\alpha\,r^2}{1+r^2}\right)^{n+1}
\end{equation}
and, together with $B_n=\frac{r}{1+r^2}\,S(r,1)$, \eqref{Euler2}.
\end{proof}

If $r\ne 0$, $\arctan r=\sgn(r)\,\frac{\pi}{2}-\arctan\frac{1}{r}$, \eqref{Euler1} with $\frac{1}{r}$ instead of $r$, and $r=\sgn(r)\abs{r}$ imply
\begin{equation*}
   \arctan r=\sgn(r)\left(\tfrac{\pi}{2}-\tfrac{\abs{r}}{1+r^2} \sum_{k=0}^\infty\tfrac{(2k)!!}{(2k+1)!!}\,\tfrac{1}{(1+r^2)^k}\right)
\end{equation*}
and \eqref{Euler2} with $B_n=\frac{(2n+2)!!}{(2n+3)!!}\,\frac{\abs{r}}{(1+r^2)^{n+2}}$. This series for $\abs{r}>1$ converges faster than \eqref{Euler1}, so we can already sense some kind of similarity with \eqref{Owen4}.

In the sequel, the lower and upper regularized gamma functions $P(a,x)=\frac{\gamma(a,x)}{\Gamma(a)}$ and $Q(a,x)=\frac{\Gamma(a,x)}{\Gamma(a)}$ respectively, where $\gamma(a,x)=\int_0^x t^{a-1}e^{-t}\,\mathrm{d}t$ and $\Gamma(a,x)=\int_x^\infty t^{a-1}e^{-t}\,\mathrm{d}t$ are the lower and upper incomplete gamma functions respectively, both defined for $a>0$ and $x\ge 0$, will be very important, as well as their lower and upper bounds.

In the Alzer's inequality
\begin{equation}\label{Alzer1}
   (1-e^{-s_{k+1}x})^a\le P(a,x)\le(1-e^{-x})^a\quad(a\ge 1,x\ge 0),
\end{equation}
where $s_{k+1}=((k+1)!)^{-\frac{1}{k+1}}$ \citetext{\citealp[p.~772, Theorem 1]{Alzer1997}, \citealp[p.~221, (5.4)]{Gautschi1998}}, allowing $a=1$ resulted in strict inequalities becoming non-strict. It implies
\begin{equation}\label{Alzer2}
   1-(1-e^{-x})^a\le Q(a,x)\le 1-(1-e^{-s_{k+1}x})^a.
\end{equation}

For $n\in\mathbb{N}_0$ and $x\ge 0$, we have $I_n(x)=e^x\,P(n+1,x)$ \citep[p.~908, 8.352, 1.]{Gradshteyn2015},
where $I_n(x)=e^x-\sum_{k=0}^n\frac{x^k}{k!}$ is the remainder of the Taylor series of the exponential function. The Lagrangian form of the remainder $I_n(x)=\frac{x^{n+1}\,e^{\xi_n(x)}}{(n+1)!}$, where $0<\xi_n(x)<x$, implies $I_n(x)\ge \frac{x^{n+1}}{(n+1)!}$ for $x\ge 0$. Since $I_n(x)<\frac{n+2}{n+2-x}\,\frac{x^{n+1}}{(n+1)!}$ for $x\in(0,n+2)$ \citep[pp.~323--324, 3.8.25]{Mitrinovic1970a}, it follows
\begin{equation}\label{Mitrinovic}
   \tfrac{x^{k+1}\,e^{-x}}{(k+1)!}\le P(k+1,x)\le\tfrac{2\,x^{k+1}\,e^{-x}}{(k+1)!} \quad\left(k\in\mathbb{N}_0,0\le x\le\tfrac{k+2}{2}\right),
\end{equation}
where the equality holds only for $x=0$.

\begin{theorem}\label{Theorem1}
Let $\hat{f}(x)=\sum_{k=0}^\infty\hat{c}_k\,x^{2k}$ and $\tilde{f}(x)=\sum_{k=0}^\infty\tilde{c}_k\,x^{2k+1}$ be absolutely convergent series for $x\in\mathbb{R}$. Then
\begin{equation}\label{Eq:Theo1a}
   \hat{F}(x)=\int_0^x\varphi(\beta t)\,\hat{f}(t)\,\mathrm{d}t=\tfrac{1}{2}\sum_{k=0}^\infty\tfrac{(2k-1)!!\, \hat{c}_k}{\abs{\beta}^{2k+1}}\,P(k+\tfrac{1}{2},\tfrac{1}{2}\beta^2x^2)
\end{equation}
and
\begin{equation}\label{Eq:Theo1b}
   \tilde{F}(x)=\int_0^x\varphi(\beta t)\,\tilde{f}(t)\,\mathrm{d}t=\tfrac{1}{\sqrt{2\pi}}\sum_{k=0}^\infty \tfrac{(2k)!!\,\tilde{c}_k}{\beta^{2k+2}}\,P(k+1,\tfrac{1}{2}\beta^2x^2)
\end{equation}
for $x\ge 0$ and $\beta\ne 0$.
\end{theorem}

\begin{proof}
Let $f_k(t)=\hat{c}_k\,\varphi(\beta t)\,t^{2k}$ and $F_k(x)=\int_0^x f_k(t)\,\mathrm{d}t$ for $x\ge 0$. Using $(2k-1)!!=\frac{2^k}{\sqrt{\pi}}\,\Gamma(k+\frac{1}{2})$ and a substitution $\frac{1}{2}\beta^2t^2=u$, we get
\begin{equation*}
   \begin{split}
      F_k(x)&=\tfrac{\hat{c}_k}{\sqrt{2\pi}}\int_0^x e^{-\frac{1}{2}\beta^2t^2}t^{2k}\,\mathrm{d}t=\tfrac{2^{k-1} \hat{c}_k}{\sqrt{\pi}\;\abs{\beta}^{2k+1}}\int_0^{\frac{1}{2}\beta^2x^2} u^{k-\frac{1}{2}}e^{-u}\,\mathrm{d}u\\
      &=\tfrac{(2k-1)!!\,\hat{c}_k}{2\,\abs{\beta}^{2k+1}}\, \tfrac{\gamma(k+\frac{1}{2},\frac{1}{2}\beta^2x^2)}{\Gamma(k+\frac{1}{2})}= \tfrac{(2k-1)!!\,\hat{c}_k}{2\,\abs{\beta}^{2k+1}}\, P(k+\tfrac{1}{2},\tfrac{1}{2}\beta^2x^2).
   \end{split}
\end{equation*}
Setting $j=\lceil\beta^2x^2\rceil+1$ implies $0\le u<\frac{j+2}{2}$ for $u\in[0,\frac{1}{2}\beta^2x^2]$. Then $\sum_{k=0}^\infty\abs{F_k(x)}= \sum_{k=0}^j\abs{F_k(x)}+\sum_{k=j+1}^\infty \abs{F_k(x)}$, where the first sum is finite. Using the fact that $P(a,x)$ is a decreasing function of $a$ \citep[p.~154]{Alzer1998}, \eqref{Mitrinovic} and the absolute convergence assumption of $\hat{f}(x)$, it follows
\begin{equation*}
   P(k+\tfrac{1}{2},\tfrac{1}{2}\beta^2x^2)\le P(k,\tfrac{1}{2}\beta^2x^2)\le \tfrac{2\,\beta^{2k}\,x^{2k}\,e^{-\frac{1}{2}\beta^2x^2}}{2^k k!}\quad(k\ge j+1)
\end{equation*}
and
\begin{equation*}
   \sum_{k=j+1}^\infty\abs{F_k(x)}=\sum_{k=j+1}^\infty\tfrac{(2k-1)!!\,
   \abs{\hat{c}_k}}{2\,\abs{\beta}^{2k+1}}\,P(k+\tfrac{1}{2},\tfrac{1}{2}\beta^2x^2)
   \le\tfrac{e^{-\frac{1}{2}\beta^2x^2}}{\abs{\beta}} \sum_{k=j+1}^\infty\tfrac{(2k-1)!!}{(2k)!!}\;\abs{\hat{c}_k}\,x^{2k}<\infty.
\end{equation*}
Noting that $\int_0^x\abs{f_k(t)}\,\mathrm{d}t=\abs{F_k(x)}$, we proved that $\sum_{k=0}^\infty\int_0^x\abs{f_k(t)}\,\mathrm{d}t<\infty$. By Fubini's theorem it follows that $\int_0^x \varphi(\beta t)\,\hat{f}(t)\,\mathrm{d}t=\sum_{k=0}^\infty F_k(x)$, hence \eqref{Eq:Theo1a}.

Let $f_k(t)=\tilde{c}_k\,\varphi(\beta t)\,t^{2k+1}$ and $F_k(x)=\int_0^x f_k(t)\,\mathrm{d}t$ for $x\ge 0$. Analogously to above, and using $(2k)!!=2^k k!=2^k\,\Gamma(k+1)$, we get $F_k(x)=\frac{(2k)!!\,\tilde{c}_k}{\sqrt{2\pi}\; \beta^{2k+2}}\,P(k+1,\frac{1}{2}\beta^2x^2)$,
\begin{equation*}
   \sum_{k=j+1}^\infty\abs{F_k(x)}\le\tfrac{\abs{x}\,e^{-\frac{1}{2}\beta^2x^2}} {\sqrt{2\pi}}\sum_{k=j+1}^\infty\tfrac{\abs{\tilde{c}_k}\, \abs{x}^{2k+1}}{k+1}<\infty
\end{equation*}
and \eqref{Eq:Theo1b}.
\end{proof}

Note that in the same way as the equations \eqref{Eq:Theo1a} and \eqref{Eq:Theo1b} were derived, analogous equations can be derived if the integration limits $0$ and $x$ are replaced by $x$ and $\infty$ respectively, and lower regularized gamma function is replaced by the upper one. However, for each case it should be checked whether the order of integration and summation can be reversed.

The usefulness of Theorem \ref{Theorem1} is that if the sequences $\{\hat{c}_k\}$ and $\{\tilde{c}_k\}$ can be computed recursively, then the terms of the series \eqref{Eq:Theo1a} and \eqref{Eq:Theo1b} can be too. Let $q=\frac{1}{2}\beta^2x^2$. Crucial is the calculation of $\hat{d}_k=Q(k+\frac{1}{2},q)$ and $\tilde{d}_k=P(k+1,q)$, $k\in\mathbb{N}_0$, the rest is trivial. Since $P(k+1,q)=1-Q(k+1,q)$, we will calculate $d_k=Q(k+1,q)$ instead of $\tilde{d}_k=1-d_k$.

Using $P(\frac{1}{2},q)=2\,\Phi(\sqrt{2q}\,)-1$ \citep[p.~934, 26.2.30]{Abramowitz1972}, we get $\hat{d}_0=1-P(\frac{1}{2},q)=2\,(1-\Phi(\sqrt{2q}\,))$, and $d_0=e^{-q}$. Since $P(a+1,x)=P(a,x)-\frac{x^a\,e^{-x}}{\Gamma(a+1)}$ \citep[p.~262, 6.5.21]{Abramowitz1972} implies $Q(a+1,x)=Q(a,x)+\frac{x^a\,e^{-x}}{\Gamma(a+1)}$, using auxiliary variables $\hat{b}_k=\frac{q^{k-\frac{1}{2}}\,e^{-q}}{\Gamma(k+\frac{1}{2})}$ and $b_k=\frac{q^k\,e^{-q}}{\Gamma(k+1)}$, it follows
\begin{equation*}
   \hat{b}_0=\tfrac{e^{-q}}{\sqrt{q\,\pi}},\quad\hat{d}_0=2\,(1-\Phi(\sqrt{2\,q}\,)) \quad\text{and}\quad
   \hat{b}_{k+1}=\tfrac{q\,\hat{b}_k}{k+\frac{1}{2}},\quad\hat{d}_{k+1}= \hat{d}_k+\hat{b}_{k+1}
\end{equation*}
for $q\ne 0$ and $k=0,1,2,\dotsc$, and
\begin{equation}\label{Eq:Theo1c}
   b_0=e^{-q},\quad d_0=b_0\quad\text{and}\quad b_{k+1}=\tfrac{q\,b_k}{k+1},\quad d_{k+1}=d_k+b_{k+1}
\end{equation}
for $k=0,1,2,\dotsc$

\section{Main results}\label{sec:Main}

The core of this article are Theorems \ref{Theorem2}, \ref{Theorem3} and \ref{Theorem4}. Since the first two do not specify variants for the transformed parameters, the last one, together with \eqref{Owen2} and \eqref{Owen3}, is more practically useful. It also justifies the title of the article.

\begin{theorem}\label{Theorem2}
Let $h\in\mathbb{R}$, $\varrho\in(-1,1)$, $r=\frac{\varrho}{\sqrt{1-\varrho^2}}$ and $q=\frac{1}{2}(1+r^2)h^2=\frac{h^2}{2\,(1-\varrho^2)}$. Then
\begin{equation}\label{Eq:Theo2a}
   \begin{split}
      \Phi_\varrho^2(h,0)&=\tfrac{1}{2}\,\Phi(h)+\tfrac{\arctan r}{2\pi}-\tfrac{r}{2\pi\,(1+r^2)}\sum_{k=0}^\infty c_k\left(\tfrac{r^2}{1+r^2}\right)^k\\
      &=\tfrac{1}{2}\,\Phi(h)+\tfrac{\arcsin\varrho}{2\pi}-\tfrac{\sqrt{1-\varrho^2}} {2\pi}\, \sum_{k=0}^\infty c_k\,\varrho^{2k+1},
   \end{split}
\end{equation}
where $c_k=\frac{(2k)!!}{(2k+1)!!}\,P(k+1,q)$. The series $S=\sum_{k=0}^\infty c_k\left(\frac{r^2}{1+r^2}\right)^k$ converges for every $r,h\in\mathbb{R}$. If it is truncated after the $n^{th}$ term, then
\begin{equation}\label{Eq:Theo2b}
   \Phi_\varrho^2(h,0)=\tfrac{1}{2}\,\Phi(h)+\tfrac{\arctan r}{2\pi}-\tfrac{r\,S_n}{2\pi\,(1+r^2)}-\sgn(r)R_n\quad(n\in\mathbb{N}_{-1}),
\end{equation}
where $S_n=\sum_{k=0}^n c_k\left(\frac{r^2}{1+r^2}\right)^k$ and $R_n=\frac{\abs{r}}{2\pi\,(1+r^2)} \sum_{k=n+1}^\infty c_k\left(\frac{r^2}{1+r^2}\right)^k$ is bounded by
\begin{equation}\label{Eq:Theo2c}
   0\le R_n\le\tfrac{B_n}{1+e^{-q}r^2}\le B_n, 
\end{equation}
where $B_n=\frac{(2n+2)!!}{(2n+3)!!}\,\frac{\abs{r}}{2\pi}\,(1-e^{-q})^{n+2} \left(\frac{r^2}{1+r^2}\right)^{n+1}$.
\end{theorem}

\begin{proof}
From \eqref{Owen1} and $\Phi(x)-\frac{1}{2}=\varphi(x)\sum_{k=0}^\infty\frac{x^{2k+1}}{(2k+1)!!}$ for $x\in\mathbb{R}$ \citep[p.~932, 26.2.11]{Abramowitz1972}, it follows
\begin{equation}\label{Eq:Theo2d}
   \tfrac{\partial T(h,r)}{\partial h}=-\,\tfrac{h}{2\pi}\int_0^r e^{-\frac{1}{2}h^2(1+x^2)}\,\mathrm{d}x=-\varphi(h)\left(\Phi(rh)-\tfrac{1}{2}\right)= \varphi(\beta h)\sum_{k=0}^\infty\tilde{c}_k\,h^{2k+1},
\end{equation}
$r,h\in\mathbb{R}$, where $\beta=\sqrt{1+r^2}$ and $\tilde{c}_k=-\,\frac{r^{2k+1}}{\sqrt{2\pi}\;(2k+1)!!}$. Since
\begin{equation*}
   \sum_{k=0}^\infty\tilde{c}_k\,h^{2k+1}=-\,\tfrac{\varphi(h)}{\varphi(\beta h)}\left(\Phi(rh)-\tfrac{1}{2}\right)= -\,e^{\frac{1}{2}r^2h^2}\left(\Phi(rh)-\tfrac{1}{2}\right),
\end{equation*}
the series on the left side converges for every $r,h\in\mathbb{R}$ and is an odd function of $rh$. If $rh<0$, then all its terms are positive, implying that it converges absolutely.

All requirements of Theorem \ref{Theorem1} are fulfilled. Assuming $r\ge 0$ and using $T(0,r)=\frac{\arctan r}{2\pi}$, \eqref{Eq:Theo2d} and \eqref{Eq:Theo1b}, it follows
\begin{equation*}
   T(h,r)=\tfrac{\arctan r}{2\pi}-\tfrac{1}{2\pi}\sum_{k=0}^\infty\tfrac{(2k)!!}{(2k+1)!!}\, \tfrac{r^{2k+1}}{\beta^{2k+2}}\,P(k+1,\tfrac{1}{2}\beta^2h^2),
\end{equation*}
which, together with \eqref{Phirh0a}, $\varphi(h)\,T_r(-\frac{1}{2}h^2)=T(h,r)$, $\frac{r^2}{\beta^2}=\frac{r^2}{1+r^2}=\varrho^2$, $\frac{r}{1+r^2}=\varrho\,\sqrt{1-\varrho^2}$ and $\frac{1}{2}\beta^2h^2=q$, imply \eqref{Eq:Theo2a} for $r\ge 0$. Since \eqref{Owen1} implies $T(h,-r)=-T(h,r)$, \eqref{Eq:Theo2a} is also valid for $r<0$. Let $\alpha=1-e^{-q}$. Using \eqref{Alzer1}, \eqref{Euler3} and \eqref{Euler4}, it follows
\begin{equation*}
   0\le R_n\le\tfrac{\alpha\,\abs{r}}{2\pi\,(1+r^2)}\sum_{k=n+1}^\infty \tfrac{(2k)!!}{(2k+1)!!}\left(\tfrac{\alpha\,r^2}{1+r^2}\right)^k= \tfrac{\alpha\,\abs{r}}{2\pi\,(1+r^2)}\,S(r,\alpha)
\end{equation*}
and \eqref{Eq:Theo2c}.
\end{proof}

Note that if $rh\ne 0$, then $c_k>0$ and $\frac{c_{k+1}}{c_k}=\frac{2k+2}{2k+3}\,\frac{P(k+2,q)}{P(k+1,q)}<1$. Hence the sequence $\left\{c_k\left(\frac{r^2}{1+r^2}\right)^k\right\}$ is decreasing.

\begin{corollary}
For every $h\in\mathbb{R}$ we have
\begin{equation}\label{Eq:Corr1a}
   \Phi(h)=\tfrac{1}{2}+\sgn(h)\sqrt{S},\quad\text{where}\quad S=\tfrac{1}{2\pi}\sum_{k=0}^\infty\tfrac{(2k)!!}{(2k+1)!!}\, \tfrac{P(k+1,h^2)}{2^k}.
\end{equation}
\end{corollary}

\begin{proof}
Using \eqref{Eq:Theo2a} with $\varrho=-\,\frac{\sqrt{2}}{2}$, hence $r=-1$, $q=h^2$, and the right equation \eqref{Phirh0b}, it follows
\begin{equation*}
   \tfrac{1}{2}\,\Phi(h)-\tfrac{1}{8}+\tfrac{1}{4\pi}\sum_{k=0}^\infty \tfrac{(2k)!!}{(2k+1)!!}\,\tfrac{P(k+1,h^2)}{2^k}=\tfrac{1}{2}(\Phi(h))^2
\end{equation*}
and a quadratic equation $(\Phi(h))^2-\Phi(h)+\frac{1}{4}-S=0$. If $h=0$, then $S=0$ and \eqref{Eq:Corr1a} is valid. If $h\ne 0$, since $\Phi(h)>\frac{1}{2}$ if $h>0$ and $\Phi(h)<\frac{1}{2}$ if $h<0$, only $\frac{1}{2}+\sgn(h)\sqrt{S}$ is a proper root of the two. Note that sending $h\to\infty$ implies $\sum_{k=0}^\infty\frac{k!}{(2k+1)!!}=\frac{\pi}{2}$.
\end{proof}

\begin{theorem}\label{Theorem3}
Let $h\in\mathbb{R}$, $\varrho\in(-1,1)$, $r=\frac{\varrho}{\sqrt{1-\varrho^2}}$ and $q=\frac{1}{2}(1+r^2)h^2=\frac{h^2}{2\,(1-\varrho^2)}$. Then
\begin{equation}\label{Eq:Theo3a}
   \begin{split}
      \Phi_\varrho^2(h,0)&=\tfrac{1}{2}\,\Phi(h)+\tfrac{r}{2\pi\,(1+r^2)} \sum_{k=0}^\infty c_k\left(\tfrac{r^2}{1+r^2}\right)^k\\
      &=\tfrac{1}{2}\,\Phi(h)+\tfrac{\sqrt{1-\varrho^2}}{2\pi}\,\sum_{k=0}^\infty c_k\,\varrho^{2k+1},
   \end{split}
\end{equation}
where $c_k=\frac{(2k)!!}{(2k+1)!!}\,Q(k+1,q)$. The series $S=\sum_{k=0}^\infty c_k\left(\frac{r^2}{1+r^2}\right)^k$ converges for every $r,h\in\mathbb{R}$. If it is truncated after the $n^{th}$ term, then
\begin{equation}\label{Eq:Theo3b}
   \Phi_\varrho^2(h,0)=\tfrac{1}{2}\,\Phi(h)+\tfrac{r\,S_n}{2\pi\,(1+r^2)}+ \sgn(r)R_n\quad(n\in\mathbb{N}_{-1}),
\end{equation}
where $S_n=\sum_{k=0}^n c_k\left(\frac{r^2}{1+r^2}\right)^k$ and $R_n=\frac{\abs{r}}{2\pi\,(1+r^2)} \sum_{k=n+1}^\infty c_k\left(\frac{r^2}{1+r^2}\right)^k$ is bounded by
\begin{equation}\label{Eq:Theo3c}
   0\le R_n\le B_n,\quad\text{where}\quad B_n=\tfrac{(2n+2)!!}{(2n+3)!!}\,\tfrac{\abs{r}}{2\pi} \left(\tfrac{r^2}{1+r^2}\right)^{n+1}.
\end{equation}
\end{theorem}

\begin{proof}
Using \eqref{Eq:Theo2a}, \eqref{Euler1}, $P(k+1,x)=1-Q(k+1,x)$, $\frac{r^2}{1+r^2}=\varrho^2$ and $\frac{r}{1+r^2}=\varrho\,\sqrt{1-\varrho^2}$, it follows \eqref{Eq:Theo3a}. Noting that $R_n=\frac{\abs{r}}{2\pi\,(1+r^2)}\,S(r,1)$ and \eqref{Euler4} imply \eqref{Eq:Theo3c}.
\end{proof}

Note that the upper bound \eqref{Alzer2} is not suitable for a simple estimation of $R_n$ here because $\lim_{k\to\infty}s_{k+1}=0$, hence $Q(k+1,q)\le 1$ was implicitly used instead. $Q(a,x)$ is an increasing function of $a$. The ratio $\frac{c_{k+1}}{c_k}= \frac{2k+2}{2k+3}\,\frac{Q(k+2,q)}{Q(k+1,q)}$ can be greater than $1$, e.g.\ if $h=3$ and $r=1$, we get $\frac{c_2}{c_1}=\frac{4\,Q(3,9)}{5\,Q(2,9)}\approx 4.04$. In Theorem \ref{Theorem3}, the sequence $\left\{c_k\left(\frac{r^2}{1+r^2}\right)^k\right\}$ is not always decreasing, but, in any case, the well-behaved $\arctan\,\abs{r}$ series \eqref{Euler1} divided by $2\pi$ is a majorant series for $\frac{\abs{r}\,S}{2\pi\,(1+r^2)}$ from Theorems \ref{Theorem2} and \ref{Theorem3}.

\begin{corollary}\label{Corollary2}
Let $x,y\in\mathbb{R}$, $x^2+y^2>0$, $\varrho\in(-1,1)$, $\varrho_x=\frac{(\varrho x-y)\,\sgn(x)}{\sqrt{x^2-2\varrho xy+y^2}}$, $\varrho_y=\frac{(\varrho y-x)\,\sgn(y)} {\sqrt{x^2-2\varrho xy+y^2}}$ and $q=\frac{x^2-2\varrho xy+y^2}{2\,(1-\varrho^2)}$. Let $c_k=\frac{(2k)!!}{(2k+1)!!}\,P(k+1,q)$ and $\tilde{c}_k=\frac{(2k)!!}{(2k+1)!!}\,Q(k+1,q)$ for $k\in\mathbb{N}_0$. Then
\begin{equation}\label{Eq:Corr2a}
   \begin{split}
      \Phi_\varrho^2(x,y)&=\tfrac{1}{2}\,(\Phi(x)+\Phi(y))+\tfrac{\arcsin\varrho_x+ \arcsin\varrho_y}{2\pi}-S-\beta\\
      &=\tfrac{1}{2}\,(\Phi(x)+\Phi(y))+\tilde{S}-\beta,
   \end{split}
\end{equation}
where
\begin{equation*}
   S=\tfrac{1}{2\pi\,\sqrt{2q}}\sum_{k=0}^\infty c_k\left(\abs{x}\varrho_x^{2k+1}+\abs{y}\varrho_y^{2k+1}\right),
\end{equation*}
$\tilde{S}$ is $S$ with $\tilde{c}_k$ instead of $c_k$, and $\beta=0$ if $xy>0$ or $xy=0$ and $x+y\ge 0$, and $\beta=\frac{1}{2}$ otherwise.
\end{corollary}

\begin{proof}
Since $q_x=\frac{x^2}{2\,(1-\varrho_x^2)}=\frac{x^2-2\varrho xy+y^2}{2\,(1-\varrho^2)}=q$ and $q_y=\frac{y^2}{2\,(1-\varrho_y^2)}=\frac{x^2-2\varrho xy+y^2}{2\,(1-\varrho^2)}=q$, $\sqrt{1-\varrho_{x\phantom{y}}^2}=\frac{\abs{x}\,\sqrt{1-\varrho^2}} {\sqrt{x^2-2\varrho xy+y^2}}=\frac{\abs{x}}{\sqrt{2q}}$ and $\sqrt{1-\varrho_y^2}=\frac{\abs{y}}{\sqrt{2q}}$, \eqref{Phirxy2}, \eqref{Phirxy3}, \eqref{Eq:Theo2a} and \eqref{Eq:Theo3a} imply \eqref{Eq:Corr2a}.
\end{proof}

Note that $\frac{r}{1+r^2}=\frac{\overline{r}}{1+\overline{r}^2}$ and $q=\frac{1}{2}(1+r^2)h^2=\frac{1}{2}(1+\overline{r}^2)\overline{h}^2=\overline{q}$, therefore they are invariant on the transformation of $rh\mapsto\overline{h}$ and $\frac{1}{r}\mapsto\overline{r}$, but $S_n$, $R_n$ and $B_n$ from Theorems \ref{Theorem2} and \ref{Theorem3} are not because $\varrho^2= \frac{r^2}{1+r^2}\mapsto\frac{1}{1+r^2}=\frac{\overline{r}^2}{1+\overline{r}^2}= \overline{\varrho}^2$ and $\varrho^2\ne\overline{\varrho}^2$ if $\abs{r}\ne 1$. They transform to $\overline{S}_n$, $\overline{R}_n$ and $\overline{B}_n$. By transforming the parameters and using \eqref{Owen5}, the equations \eqref{Eq:Theo2b} and \eqref{Eq:Theo3b} can be rewritten to
\begin{equation}\label{Eq:Theo2e}
   \Phi_\varrho^2(h,0)=\tfrac{1}{2}\,\Phi(rh)+\Phi(h)(1-\Phi(rh))- \tfrac{\arctan\overline{r}}{2\pi}+\tfrac{r\,\overline{S}_n}{2\pi\,(1+r^2)}+ \sgn(r)\overline{R}_n-\beta
\end{equation}
and
\begin{equation}\label{Eq:Theo3d}
   \Phi_\varrho^2(h,0)=\tfrac{1}{2}\,\Phi(rh)+\Phi(h)(1-\Phi(rh))- \tfrac{r\,\overline{S}_n}{2\pi\,(1+r^2)}-\sgn(r)\overline{R}_n-\beta
\end{equation}
respectively, where $\beta=0$ if $r\ge 0$ and $\beta=\frac{1}{2}$ if $r<0$.

The transformation makes sense if $\abs{r}>1$, implying $\overline{\varrho}^2<\frac{1}{2}$ and faster convergence. For example, the inequality $0\le R_n\le B_n$ from Theorem \ref{Theorem3} transforms into $0\le\overline{R}_n\le\overline{B}_n$, where the upper bound $\overline{B}_n$ is $r^{2n+4}$-times smaller than a corresponding upper bound $B_n$. The noteworthy cost for faster convergence is the additional computation of $\Phi(rh)$.

\begin{theorem}\label{Theorem4}
If $\abs{r}<\infty$, the Owen's $T$ function \eqref{Owen1} can be written as
\begin{subequations}
   \begin{align}
      T(h,r)
      &=\tfrac{\arctan r}{2\pi}-\tfrac{r}{2\pi\,(1+r^2)}\sum_{k=0}^\infty\tfrac{(2k)!!}{(2k+1)!!}\, P(k+1,q)\left(\tfrac{r^2}{1+r^2}\right)^k\label{Owen6a}\\
      &=\tfrac{r}{2\pi\,(1+r^2)}\sum_{k=0}^\infty\tfrac{(2k)!!}{(2k+1)!!}\, Q(k+1,q)\left(\tfrac{r^2}{1+r^2}\right)^k\label{Owen6b}\\
      &=U(h,r)-\tfrac{\arctan\frac{1}{r}}{2\pi}+\tfrac{r}{2\pi\,(1+r^2)} \sum_{k=0}^\infty\tfrac{(2k)!!}{(2k+1)!!}\,P(k+1,q) \left(\tfrac{1}{1+r^2}\right)^k\label{Owen6c}\\
      &=U(h,r)-\tfrac{r}{2\pi\,(1+r^2)}\sum_{k=0}^\infty\tfrac{(2k)!!}{(2k+1)!!}\, Q(k+1,q)\left(\tfrac{1}{1+r^2}\right)^k,\label{Owen6d}
   \end{align}
\end{subequations}
where $q=\frac{1}{2}(1+r^2)h^2$, $U(h,r)=\frac{1}{2}(\Phi(h)+\Phi(rh))-\Phi(h)\,\Phi(rh)-\beta$, and $\beta=0$ if $r\ge 0$ and $\beta=\frac{1}{2}$ if $r<0$.
\end{theorem}
\begin{proof}
Using \eqref{Phirh0a} and $\varphi(h)\,T_r(-\frac{1}{2}h^2)=T(h,r)$, the equations \eqref{Eq:Theo2a} and \eqref{Eq:Theo3a} imply \eqref{Owen6a} and \eqref{Owen6b} respectively, which, together with \eqref{Owen4}, imply \eqref{Owen6c} and \eqref{Owen6d} respectively.
\end{proof} 

Note that when truncating the series \eqref{Owen6a} and \eqref{Owen6b}, the truncation error in absolute value can be estimated by \eqref{Eq:Theo2c} and \eqref{Eq:Theo3c} respectively. The same upper bounds with $\overline{r}=\frac{1}{r}$ instead of $r$ also apply to \eqref{Owen6c} and \eqref{Owen6d}.

Because of \eqref{Owen6b} and $0\le Q(k+1,q)\le 1$, the Owen's $T$ function can be seen as a modified Euler's arctangent series \eqref{Euler1} divided by $2\pi$, which is its majorant series.

\section{Recursion and asymptotic speed of convergence}\label{sec:Recursion}

\floatname{algorithm}{Recursion}

Let $\tilde{S}=\frac{r\,S}{2\pi\,(1+r^2)}$ and $\tilde{e}_k=B_k$, where $S$ and $B_k$ relate to Theorem \ref{Theorem3}. Using \eqref{Eq:Theo3a} and \eqref{Eq:Theo1c}, the series $\tilde{S}$ can be calculated by Recursion \ref{Recurs1} (tildes are omitted). Analogously, let $\hat{S}=\frac{r\,S}{2\pi\,(1+r^2)}$ and $\hat{e}_k=B_k$, where $S$ and $B_k$ relate to Theorem \ref{Theorem2}. The series $\hat{S}$ can be calculated by Recursion \ref{Recurs2} (hats are omitted). In practice, both recursions can be easily combined, since only four steps, marked with $\triangleright$, are different.

In both cases, given $\epsilon>0$, the recursion is terminated when the condition $e_k<\epsilon$ is met and $S_k$ is taken as its result. However, if $\epsilon>0$ is too small, or even intentionally negative, the calculation proceeds up to the actual accuracy threshold -- the recursion is terminated when the condition $S_k\le S_{k-1}$ is met. This criterion was used in the calculations presented in the next section.

In both recursions, all variables are positive and are bounded upward. If $q\ne 0$, the sequence $\{d_k\}$ strictly increases and, due to \eqref{Alzer2}, converges to $1$. The sequence $\{S_k\}$ is strictly increasing if $r\ne 0$. If $r=0$, the recursion terminates with $S_1=0$. Note that \eqref{Owen1} implies $\abs{T(h,r)}\le\frac{\arctan\,\abs{r}}{2\pi}$. If step $7$ in Recursion \ref{Recurs2} is replaced by $S_0\leftarrow a_0\,(1-d_0)-\frac{\arctan\,\abs{r}}{2\pi}$, the sequence $\{S_k\}$ changes its sign but keeps the direction. However, the modified Recursion \ref{Recurs2} computes $\hat{S}=\frac{r\,S}{2\pi\,(1+r^2)}-\frac{\arctan r}{2\pi}$.
\begin{center}
\begin{minipage}[t]{0.48\textwidth}
\begin{algorithm}[H]
\caption{}\label{Recurs1}
\begin{algorithmic}[1]
   \Require $h,r,\epsilon\in\mathbb{R}$
   \State $p\gets\frac{r^2}{1+r^2}$
   \State $q\gets\frac{1}{2}(1+r^2)h^2$
   \State $g\gets p$ \phantom{$(1-e^{-q})$} \Comment
   \State $a_0\gets\frac{\abs{r}}{2\pi\,(1+r^2)}$
   \State $b_0\gets e^{-q}$
   \State $d_0\gets b$
   \State $S_0\gets a\,d_0$ \phantom{$()$} \Comment
   \State $e_0\gets\frac{\abs{r}\,p}{3\pi}$ \Comment
   \State $k\gets 0$
   \Repeat
      \State $a_{k+1}\gets\frac{(2k+2)\,p\,a_k}{2k+3}$
      \State $b_{k+1}\gets\frac{q\,b_k}{k+1}$
      \State $d_{k+1}\gets d_k+b_{k+1}$
      \State $S_{k+1}\gets S_k+a_{k+1}\,d_{k+1}$ \phantom{$()$} \Comment
      \State $e_{k+1}\gets\frac{(2k+4)\,g\,e_k}{2k+5}$
      \State $k\gets k+1$
   \Until{$e_k<\epsilon\textbf{ or }S_k\le S_{k-1}$}
   \State $S_k\gets\sgn(r)\,S_k$
\end{algorithmic}
\end{algorithm}
\end{minipage}\hspace{0.03\textwidth}
\begin{minipage}[t]{0.48\textwidth}
\begin{algorithm}[H]
\caption{}\label{Recurs2}
\begin{algorithmic}[1]
   \Require $h,r,\epsilon\in\mathbb{R}$
   \State $p\gets\frac{r^2}{1+r^2}$
   \State $q\gets\frac{1}{2}(1+r^2)h^2$
   \State $g\gets(1-e^{-q})\,p$
   \State $a_0\gets\frac{\abs{r}}{2\pi\,(1+r^2)}$
   \State $b_0\gets e^{-q}$
   \State $d_0\gets b$
   \State $S_0\gets a\,(1-d_0)$
   \State $e_0\gets\frac{\abs{r}\,(1-d_0)\,g}{3\pi}$
   \State $k\gets 0$
   \Repeat
      \State $a_{k+1}\gets\frac{(2k+2)\,p\,a_k}{2k+3}$
      \State $b_{k+1}\gets\frac{q\,b_k}{k+1}$
      \State $d_{k+1}\gets d_k+b_{k+1}$
      \State $S_{k+1}\gets S_k+a_{k+1}\,(1-d_{k+1})$
      \State $e_{k+1}\gets\frac{(2k+4)\,g\,e_k}{2k+5}$
      \State $k\gets k+1$
   \Until{$e_k<\epsilon\textbf{ or }S_k\le S_{k-1}$}
   \State $S_k\gets\sgn(r)\,S_k$
\end{algorithmic}
\end{algorithm}
\end{minipage}
\end{center}

Let $\tilde{S}_n$ and $\hat{S}_n$ be the results of Recursion \ref{Recurs1} and modified Recursion \ref{Recurs2} respectively, calculated with the same $h$ and $r$ if $\abs{r}\le 1$, or with $\overline{h}$ and $\overline{r}$ if $\abs{r}>1$. Let $\tilde{\Phi}_\varrho^2(h,0)$ be calculated by \eqref{Eq:Theo3b} or \eqref{Eq:Theo3d}, where unknown remainders are ignored. Analogously, let $\hat{\Phi}_\varrho^2(h,0)$ be calculated by \eqref{Eq:Theo2b} or \eqref{Eq:Theo2e}, where also $\frac{\arctan r}{2\pi}$ is ignored. If $\abs{r}\le 1$, then $\tilde{\Phi}_\varrho^2(h,0)\le\Phi_\varrho^2(h,0)\le\hat{\Phi}_\varrho^2(h,0)$ if $r\ge 0$ and $\tilde{\Phi}_\varrho^2(h,0)\ge\Phi_\varrho^2(h,0)\ge\hat{\Phi}_\varrho^2(h,0)$ if $r<0$. Even if $\abs{r}>1$, one of $\tilde{\Phi}_\varrho^2(h,0)$ and $\hat{\Phi}_\varrho^2(h,0)$ underestimates and the other overestimates $\Phi_\varrho^2(h,0)$. Consequently, the upper bounds \eqref{Eq:Theo2c} and \eqref{Eq:Theo3c} are not needed if $\tilde{S}$ and modified $\hat{S}$ are calculated in the same recursion which terminates when $\abs{\hat{S}_k-\tilde{S}_k}$ become small enough. Upper bounds are also not needed if in Recursions \ref{Recurs1} and \ref{Recurs2} step $17$ is replaced so that $S_k\le S_{k-1}$ is the only termination condition. In this case, steps $3$, $8$ and $15$ are unnecessary.

Given $q$, which one of the series \eqref{Eq:Theo2a} and \eqref{Eq:Theo3a} converges faster? Since $P(k+1,x)$ is a cdf of the gamma distribution with the median $\nu\in(k+\frac{2}{3},k+1)$, it follows $P(k+1,q)\le Q(k+1,q)$ if $q\le\nu<k+1$. The answer depends on the number of considered terms, which depends on the prescribed accuracy. On this basis, it seems that series \eqref{Eq:Theo2a} is more suitable than \eqref{Eq:Theo3a} for ``small'' $q$ and vice versa for ``large'' $q$. Tables \ref{Table:1} and \ref{Table:4} in the next section show how this is reflected in practice.

Gautschi's inequality $x^{1-s}<\frac{\Gamma(x+1)}{\Gamma(x+s)}<(x+1)^{1-s}$ \citep[p.~138, 5.6.4]{Olver2010} with $s=\frac{1}{2}$ implies $\frac{\sqrt{k\,\pi}}{2k+1}< \frac{(2k)!!}{(2k+1)!!}<\frac{\sqrt{(k+1)\,\pi}}{2k+1}$. Combining it with \eqref{Mitrinovic}, sending $k\to\infty$ and neglecting non-essential constants imply that the series \eqref{Eq:Theo2a} asymptotically converges as $\sum_k\frac{q^{k+1}e^{-q}}{(k+1)!}\,\frac{\varrho^{2k}}{\sqrt{k+1}}$. Assuming calculation with $\overline{h}$, $\overline{r}$ and \eqref{Eq:Theo2e} instead of $h$, $r$ and \eqref{Eq:Theo2b} respectively if $\abs{r}>1$, hence $\varrho^2>\frac{1}{2}$, and recalling that $q$ is invariant on the transformation of $rh\mapsto\overline{h}$ and $\frac{1}{r}\mapsto\overline{r}$, we get an asymptotic convergence as $\sum_k\frac{\overline{q}^{k+1} e^{-\overline{q}}}{(k+1)!}\,\frac{\overline{\varrho}^{2k}}{\sqrt{k+1}}= \sum_k\frac{q^{k+1}e^{-q}}{(k+1)!}\,\frac{(1-\varrho^2)^k}{\sqrt{k+1}}$. The series \eqref{Eq:Theo2a} or its corresponding transformed version, unwritten and based on \eqref{Owen6c}, asymptotically converges as $\sum_k\frac{q^{k+1}e^{-q}}{(k+1)!}\,\frac{(\min\{\varrho^2,1-\varrho^2\})^k} {\sqrt{k+1}}=\sum_k\frac{q^{k+1}e^{-q}}{(k+1)!}\, \frac{\min\{r^{2k},1\}}{\sqrt{k+1}\;(1+r^2)^k}$, where $q=\frac{1}{2}(1+r^2)h^2=\frac{h^2}{2\,(1-\varrho^2)}$. On the other hand, since \eqref{Alzer2} implies $\lim_{k\to\infty}Q(k+1,x)=1$, the series \eqref{Eq:Theo3a} or its corresponding transformed version, unwritten and based on \eqref{Owen6d}, asymptotically converges as $\sum_k\frac{(\min\{\varrho^2,1-\varrho^2\})^k}{\sqrt{k+1}}= \sum_k\frac{\min\{r^{2k},1\}}{\sqrt{k+1}\;(1+r^2)^k}$, hence in the worst case as $\sum_k\frac{1}{2^k\,\sqrt{k+1}}$.

Based on both asymptotic convergences, comparing \eqref{Eq:Theo3a} to \eqref{Eq:Theo2a} would be unfair because the calculation of $\arctan\,\abs{r}$ should also be taken into account for \eqref{Eq:Theo2a}. Since it is implicitly embedded in Recursion \ref{Recurs1}, it is calculated on the fly at the cost of worse speed of convergence. Each of the recursions is executable with the basic four arithmetic operations, two calls of elementary functions (if the calculation of $r$ from $\varrho$ is also considered), and $\arctan\abs{r}$ in the modified Recursion \ref{Recurs2}. However, the computational cost of $\Phi(h)$ (and $\Phi(rh)$ if needed) should also be taken in account, but it is the same in \eqref{Eq:Theo2a} and \eqref{Eq:Theo3a}.

\section{Numeric issues and testing results}\label{sec:Results}

A common cause of numerical instability, where two absolutely large values give an absolutely small result when added or subtracted, cannot occur in Recursions \ref{Recurs1} and \ref{Recurs2}. The calculation by them is stable, but something may happen that needs to be pointed out. The side effect of transforming $rh\mapsto\overline{h}$ and $\frac{1}{r}\mapsto\overline{r}$ is that $\abs{\overline{h}}$ can become unusually large. When computing $\Phi(h)$, according to \citet[p.~2]{Marsaglia2004}, ``[...] the region of interest may not be the statistician's $0 < x < 5$, say, but rather $10 < x < 14$ [...]''. Assume the same needs when computing $\Phi_\varrho^2(h,0)$, and using transformed parameters if $\varrho^2>\frac{1}{2}$. If $h=14$ and $\varrho=\frac{10}{\sqrt{1+10^2}}\approx 0.995$, then $r\approx 10$, $\overline{r}\approx 0.1$ and $\overline{h}\approx 140$, implying that the region of interest could be $\abs{h}<140$. Even the usual $\abs{h}<5$ transformed to $\abs{\overline{h}}<50$ is well beyond the established limits.

The ``problem'' of too large $\abs{h}$ is mentioned because we have to be aware of it if we are going to perform high-precision computation. In any case, the recursion must be complemented by a suitable computation of $\Phi(h)$ and $\arctan r$ if needed. As an example that not all are like that, the computation by the well known series $\Phi(x)=\frac{1}{2}+\frac{1}{\sqrt{2\pi}}\sum_{k=0}^\infty \frac{(-1)^k\,x^{2k+1}}{2^k k!\,(2k+1)}$ \citep[p.~932, 26.2.10]{Abramowitz1972} and the standard double precision arithmetic is numerically unstable even for moderately large $\abs{h}$, e.g.\ using the programming language \proglang{R} \citep{R2023}, for $h=10$ we get a completely wrong result, approximately $-1683$. A similar problem can lead to a significant decrease in accuracy if $T(h,r)$ is calculated with \eqref{Owen1}, as will be seen in Subsection \ref{subsec:Results1}.

In double precision arithmetic, the minimal positive non-zero number is $\eta\approx 2.23\cdot 10^{-308}$. Assuming $b_0=e^{-q}\ge\eta$, it follows $q_{max}=-\log(\eta)\approx 708.40$ and $h_{max}(r)=\sqrt{-\,\frac{2\log(\eta)}{1+r^2}}$. If $\abs{r}\le 1$, then $h_{max}(r)$ is between $h_{max}(1)\approx 26.62$ and $h_{max}(0)\approx 37.64$. If $e^{-q}<\eta$, then all elements of the sequence $\{b_k\}$ are zero. Consequently, all elements of $\{d_k\}$ are one and all elements of $\{S_k\}$ in Recursion \ref{Recurs2} are zero.

From the integral in \eqref{Owen1}, it follows $\abs{T(h,r)}\le\varphi(h)\,\frac{\Phi(\abs{r}\abs{h})-\frac{1}{2}}{\abs{h}}\le \frac{\varphi(h)}{2\,\abs{h}}$ if $h\ne 0$, where the upper bound is a decreasing function of $\abs{h}$. If $e^{-q}<\eta$ and $\abs{r}\le 1$, then $\abs{h}>h_{max}(1)$ and $\abs{T(h,r)}\le\frac{\varphi(h_{max}(1))}{2\,h_{max}(1)}\approx 1.12\cdot 10^{-156}$. Analogously to the case when $r=0$, the recursion terminates with $S_1=0$, which is practically a correct result for $T(h,r)$, as well as $\Phi_\varrho^2(h,0)=\frac{1}{2}$ if $h>0$, and $\Phi_\varrho^2(h,0)=0$ if $h<0$.

The Recursions \ref{Recurs1} and \ref{Recurs2} with adjustment of their results according to \eqref{Owen6a}--\eqref{Owen6d} are implemented in the package \pkg{Phi2rho} \citep{Komelj2023a}. It contains the functions \code{OwenT()}, hereafter \code{Phi2rho::OwenT()}, and \code{Phi2xy()}, which compute $T(h,r)$ and $\Phi_\varrho^2(x,y)$ respectively. They also enable computation with the tetrachoric and Vasicek's series, original or accelerated. The latter means that $h$ and $r$ are transformed if $r>1$ and $r<1$ for the tetrachoric and Vasicek's series respectively, and the equation \eqref{Owen5} is used, so \eqref{TetrachC} instead of \eqref{TetrachB} in the case of tetrachoric series. For the sake of comparison, all series were considered during testing in Subsection \ref{subsec:Results1}. In Subsection \ref{subsec:Results2}, the original tetrachoric and Vasicek's series were excluded because they converge too slowly for unfavorable $\varrho$.

During testing, the values of $T(h,r)$ and $\Phi_\varrho^2(x,y)$ were computed for a selected and randomly generated set of parameters. Reference values were computed using the accelerated tetrachoric series and quadruple ($128$-bit) precision. Absolute differences between tested and reference values are declared as absolute errors. The maximum absolute error is considered a measure of accuracy. Since the mean values and some quantiles are also important, they are shown in the tables in Subsection \ref{subsec:Results2}. An integral part of the testing was also a comparison of the results obtained with the \pkg{Phi2rho} package and competing packages available on CRAN.

\subsection{Numeric issues and results when computing $T(h,r)$}\label{subsec:Results1}

Numerical calculations of $T(h,r)$ were performed on $39,999$ grid points $(h,r)$, where $h$ ranges from $-10$ to $10$ in increments of $0.1$, $r=\frac{\varrho}{\sqrt{1-\varrho^2}}$ and $\varrho$ ranges from $-0.99$ to $0.99$ in increments of $0.01$. Since for $\varrho=\pm 0.99$ we get $r\approx\pm 7$, the actual range of $h$ for which $\Phi(h)$ must be computed is approximately from $-70$ to $70$, so well beyond the $h_{max}(0)\approx 37.64$.

Note \emph{atanExt = no} in tables means that Recursion \ref{Recurs1} was executed, hence \eqref{Owen6b} was used if $\abs{r}\le 1$ and \eqref{Owen6d} if $\abs{r}>1$, and \emph{atanExt = yes} means that Recursion \ref{Recurs2} was executed, hence \eqref{Owen6a} was used if $\abs{r}\le 1$ and \eqref{Owen6c} if $\abs{r}>1$. Both recursions are well suited for vectors, so each of the parameters $h$ and $r$ can be a scalar or a vector. If they are both vectors, they must be of the same length, if one of $h$ and $r$ is a vector and the other is a scalar, the latter is replicated to the length of the former. Since the number of iterations is determined by the ``worst'' component, we would expect execution to be faster if function calls are done in a loop for each component of the vector separately, but usually the run time is significantly increased due to the overhead of the loop, more initialization, etc.

For the consistency test, $T(h,r)$ was calculated on the test grid with $h$ and $r$ both scalars, one a scalar and the other a vector. For all three calculations with Recursion \ref{Recurs1}, the results matched perfectly, as well as for all three with Recursion \ref{Recurs2}. The maximum absolute difference between calculations with Recursions \ref{Recurs1} and \ref{Recurs2} was approximately $9.49\cdot 10^{-17}$.

When testing accuracy, the tetrachoric and Vasicek's series, original and accelerated, were also tested, using $T(h,r)=\Phi_\varrho^2(h,0)-\frac{1}{2}\Phi(h)$. By transforming the parameters, if appropriate, both comparison methods gained a lot, as can be concluded from Table \ref{Table:1}. The absolute errors, the maximum values of which are collected in this table, were calculated against the reference values computed with the accelerated tetrachoric series and $128$-bit precision.
\begin{table}[h]
{\footnotesize
\caption{Double precision computation of $T(h,r)=\Phi_\varrho^2(h,0)-\frac{1}{2}\Phi(h)$ on the test grid\label{Table:1}} 
\begin{center}
\begin{tabular}{lrrcrrcrr}
\toprule
\multicolumn{1}{l}{}&\multicolumn{2}{c}{\textbf{Tetrachoric}}&\multicolumn{1}{c}{}& \multicolumn{2}{c}{\textbf{Vasicek}}&\multicolumn{1}{c}{}&\multicolumn{2}{c}{\textbf{Novel series}}\tabularnewline
\cmidrule(lr){2-3} \cmidrule(lr){5-6} \cmidrule(lr){8-9}
\multicolumn{1}{l}{}&\multicolumn{2}{c}{accelerated}&\multicolumn{1}{c}{}& \multicolumn{2}{c}{accelerated}&\multicolumn{1}{c}{}&\multicolumn{2}{c}{atanExt}\tabularnewline
\multicolumn{1}{l}{}&\multicolumn{1}{c}{no}&\multicolumn{1}{c}{yes}&\multicolumn{1}{c}{}&\multicolumn{1}{c}{no}&\multicolumn{1}{c}{yes}&\multicolumn{1}{c}{}&\multicolumn{1}{c}{no}&\multicolumn{1}{c}{yes}\tabularnewline
\midrule
{\bfseries $h$ is a vector}&&&&&&&&\tabularnewline
~~Max abs. error\phantom{ab}&9.42e-16&1.13e-16&&9.37e-14&1.68e-16&&9.68e-17&1.43e-16\tabularnewline
~~Average iter. 1\phantom{ab}&123.2&47.2&&3,575.4&65.3&&66.9&22.8\tabularnewline
\midrule
{\bfseries $r$ is a vector}&&&&&&&&\tabularnewline
~~Max abs. error \phantom{ba}&4.77e-15&1.13e-16&&9.37e-14&1.68e-16&&9.68e-17&1.43e-16\tabularnewline
~~Average iter. 1 \phantom{ba}&1,601.3&65.2&&188,834.8&66.6&&75.5&36.4\tabularnewline
\midrule
{\bfseries for both}&&&&&&&&\tabularnewline
~~Average iter. 2 \phantom{ba}&84.2&29.8&&3,355.8&31.2&&38.5&18.6\tabularnewline
~~Maximum iter. \phantom{ba}&2,447&113&&199,790&141&&125&50\tabularnewline
\bottomrule
\end{tabular}
\end{center}}
\end{table}
Since the vector computations were performed, there are two average numbers of iterations. The first is empirical because the iterations were terminated according to the ``worst'' component of the vector, and the second is hypothetical, if the computations were performed for each component separately. Since all iterations ran to the accuracy limit and with an unbounded number of iterations, the original Vasicek's series needed the enormous number of iterations, not only for the worse case, but also in average, when calculating with scalar $h$ and vector $r$. However, it is not intended for cases with $\varrho^2<\frac{1}{2}$, and its error can be well estimated, which was ignored here.

Comparative computations with competing packages were also performed on the test grid, using $T(h,r)=\Phi_\varrho^2(h,0)-\frac{1}{2}\Phi(h)$. The \code{pbinorm()} function from the \pkg{VGAM} package, i.e.\ using \eqref{Owen1}, had an average and maximum absolute error of $1.12\cdot 10^{-8}$ and $3.62\cdot 10^{-7}$ respectively. Similar maximum absolute errors were obtained by the function \code{pbvnorm()} from the \pkg{pbv} package \citep{Robitzsch2020}, which is based on \citep{Drezner1990}, and \code{pmvnorm()} from the \pkg{mvtnorm} package, but only with the parameter \code{algorithm = Miwa}. 

\subsection{Numeric issues and results when computing \texorpdfstring{$\Phi_\varrho^2(x,y)$}{Phi2(x,y)}}\label{subsec:Results2}

In all $\Phi_\varrho^2(x,y)$ computational methods based on \eqref{Phirxy2} and \eqref{Phirxy3}, or \eqref{Owen2} and \eqref{Owen3}, exceptional cases with $\abs{\varrho}=1$, $x\ne 0$ and $x-y\,\sgn(\varrho)=0$ are not problematic because \eqref{Phi1xy} can be used, however, nearby cases are their weak point. Recalling that $q_x=q_y=\frac{x^2-2\varrho xy+y^2}{2\,(1-\varrho^2)}=q$ and using \eqref{Phirxy4} and \eqref{phirxy}, it follows $\frac{\partial}{\partial\varrho}\, \Phi_\varrho^2(x,y)=\frac{e^{-q}}{2\pi\,\sqrt{1-\varrho^2}}$ and $\lim_{\abs{\varrho}\to 1}\frac{\partial}{\partial\varrho}\,\Phi_\varrho^2(x,y)=0$, except if $x-y\,\sgn(\varrho)=0$. If $\abs{\varrho}$ is close to $1$ and $x-y\,\sgn(\varrho)\simeq 0$, then $x^2-2\varrho xy+y^2\simeq 0$ and $q$ may be small enough that $\frac{\partial}{\partial\varrho}\,\Phi_\varrho^2(x,y)$ is moderate or even large. As a result, even a small rounding error in the calculation of key variables that depend on $\varrho$ can cause a much larger error in the result. The calculation of $r_x$ and $r_y$ with \eqref{Owen3} is particularly problematic. In a similar situation, \citet[p.~5]{Meyer2013b} recommends an equivalent but less sensitive calculation by $r_x=\frac{x-y}{x\,\sqrt{1-\varrho^2}}- \sqrt{\frac{1-\varrho}{1+\varrho}}$ if $\varrho\approx 1$ and $r_x=\frac{x+y}{x\,\sqrt{1-\varrho^2}}-\sqrt{\frac{1+\varrho}{1-\varrho}}$ if $\varrho\approx -1$, and analogously for $r_y$. In our case, his proposal improves accuracy, but does not solve the problem sufficiently.

The problem could be solved by a hybrid computation along the lines of some other methods, e.g.\ the series from \citep[pp.~242--245]{Fayed2014a} could be used if $\abs{\varrho}\approx 1$ and $x-y\,\sgn(\varrho)\simeq 0$. Using the initial idea to develop these series, i.e.\ axis rotation in a way that removes the cross product term $-2\varrho st$ in \eqref{Phirxy1}, by setting $s=\frac{u+v}{\sqrt{2}}$ and $t=\frac{u-v}{\sqrt{2}}$ the equation \eqref{Phirxy1} transforms to
\begin{equation*}
   \begin{split}
      \Phi_\varrho^2(x,y)
      &=\tfrac{1}{2\pi\,\sqrt{1-\varrho^2}}\int_{-\infty}^\frac{x-y}{\sqrt{2}} e^{-\frac{u^2}{2\,(1-\varrho)}}\left(\int_{-\infty}^{u+\sqrt{2}\,y} e^{-\frac{v^2}{2\,(1+\varrho)}}\, \mathrm{d}v\right)\mathrm{d}u\\
      &+\tfrac{1}{2\pi\,\sqrt{1-\varrho^2}}\int_\frac{x-y}{\sqrt{2}}^\infty e^{-\frac{u^2}{2\,(1-\varrho)}}\left(\int_{-\infty}^{-u+\sqrt{2}\,x} e^{-\frac{v^2}{2\,(1+\varrho)}}\, \mathrm{d}v\right)\mathrm{d}u
   \end{split}
\end{equation*}
and, using $\frac{u}{\sqrt{1-\varrho}}=\tilde{u}$ and $\frac{v}{\sqrt{1+\varrho}}=\tilde{v}$, to
\begin{equation*}
   \begin{split}
      \Phi_\varrho^2(x,y)&=\int_{-\infty}^\frac{x-y}{\sqrt{2\, (1-\varrho)}}\varphi(\tilde{u})\,\Phi \left(\tilde{u}\sqrt{\tfrac{1-\varrho}{1+\varrho}}+ y\sqrt{\tfrac{2}{1+\varrho}}\,\right) \mathrm{d}\tilde{u}\\
      &+\int_\frac{x-y}{\sqrt{2\,(1-\varrho)}}^\infty\varphi(\tilde{u})\, \Phi\left(-\tilde{u}\sqrt{\tfrac{1-\varrho}{1+\varrho}}+ x\sqrt{\tfrac{2}{1+\varrho}}\,\right) \mathrm{d}\tilde{u}.
   \end{split}
\end{equation*}
By transforming $\tilde{u}\mapsto-\tilde{u}$ in the second integral and using \citep[p.~402, 10,010.1]{Owen1980}, it follows $\Phi_\varrho^2(x,y)=\Phi_{\varrho^*}^2\Big(\frac{x-y}{\sqrt{2\,(1-\varrho)}},y\Big)+ \Phi_{\varrho^*}^2\Big(-\frac{x-y}{\sqrt{2\,(1-\varrho)}},x\Big)$, where $\varrho^*=-\,\sqrt{\frac{1-\varrho}{2}}$. This equation is suitable for $\varrho\approx 1$. From it and $\Phi_\varrho^2(x,y)=\Phi(x)-\Phi_{-\varrho}^2(x,-y)$ we obtain the equation 
\begin{equation*}
   \Phi_\varrho^2(x,y)=\Phi(x)-\Phi_{\hat{\varrho}}^2\left(\tfrac{x+y} {\sqrt{2\,(1+\varrho)}},-y\right)-\Phi_{\hat{\varrho}}^2 \left(-\tfrac{x+y}{\sqrt{2\,(1+\varrho)}},x\right),
\end{equation*}
where $\hat{\varrho}=-\,\sqrt{\frac{1+\varrho}{2}}$, that is suitable for $\varrho\approx-1$. Both are combined in
\begin{equation}\label{Phirxy6}
   \Phi_\varrho^2(x,y)=\tfrac{1-\sgn(\varrho)}{2}\,\Phi(x)+\sgn(\varrho)\left(
   \Phi_{\tilde{\varrho}}^2(z,\tilde{y})+\Phi_{\tilde{\varrho}}^2(-z,x)\right),
\end{equation}
where $\tilde{\varrho}=-\,\sqrt{\frac{1-\abs{\varrho}}{2}}$, $\tilde{y}=y\,\sgn(\varrho)$ and $z=\frac{x-\tilde{y}}{\sqrt{2\,(1-\abs{\varrho})}}$. Noting that
\begin{equation*}
   \tilde{q}=\tfrac{\tilde{y}^2-2\tilde{\varrho}\tilde{y}z+z^2} {2\,(1-\tilde{\varrho}^2)}= \tfrac{x^2+2\tilde{\varrho}xz+z^2}{2\,(1-\tilde{\varrho}^2)}=\tfrac{x^2-2\varrho xy+y^2}{2\,(1-\varrho^2)}=q,
\end{equation*}
hence $\varphi_{\tilde{\varrho}}^2(z,\tilde{y})=\varphi_{\tilde{\varrho}}^2(-z,x)= \frac{e^{-q}}{\pi\,\sqrt{2\,(1+\abs{\varrho}})}$, we stop here because one potentially critical case is replaced with two non-critical ones.

If we continued to simplify the right side of \eqref{Phirxy6} by \eqref{Owen2}, we would get \eqref{Owen2}. This actually happens when calculating with the functions used in tests, but it turns out that the results are more accurate than if we start with \eqref{Owen2} if $\abs{\varrho}\approx 1$ and $x-y\,\sgn(\varrho)\simeq 0$. Of course, the equation \eqref{Phirxy6} could be used as a starting point instead of \eqref{Owen2} in all cases, but this is not rational because it requires $4$ calculations of the $T$ function instead of $2$, two of which cancel each other out, at least in theory, and as a result, in non-critical cases the error may be larger than when starting with \eqref{Owen2}. When testing, \eqref{Phirxy6} was used as a starting point instead of \eqref{Owen2} if $\varphi_\varrho^2(x,y)>1$. However, since $\min\left\{\abs{\varrho},\sqrt{\frac{1-\abs{\varrho}}{2}}\right\}\le\frac{1}{2}$, using \eqref{Phirxy6} may also make sense in some other circumstances, e.g.\ when computing directly with \eqref{TetrachA}.

In this case, testing was performed on the randomly generated set of parameters. After setting the random generator seed to $123$ to allow replication and independent verification, one million $x$ values, one million $y$ values, and one million $\varrho$ values were drawn from a uniform distribution on the intervals $(-10,10)$, $(-10,10)$, and $(-1,1)$ respectively.

For each triplet $(x,y,\varrho)$, auxiliary values
$q=\frac{x^2-2\varrho xy+y^2}{2\,(1-\varrho^2)}$ and $\varphi_\varrho^2(x,y)=\frac{e^{-q}}{2\pi\,\sqrt{1-\varrho^2}}$ were calculated. For the new series, the tested $\Phi_\varrho^2(x,y)$ values were computed using \eqref{Phirxy6} as a branch point and \eqref{Owen2} for both branches if $\varphi_\varrho^2(x,y)>1$, and using \eqref{Owen2} directly otherwise. The same applies for the tetrachoric and Vasicek's series, only \eqref{Phirxy2} must be used instead of \eqref{Owen2}. However, the described procedure is built in the \code{Phi2xy()} function. The tested $\Phi_\varrho^2(x,y)$ values needed for the preparation of Table \ref{Table:2} were computed with:
\begin{quote}
\begin{compactitem}
   \item[row 1:] \code{Phi2xy( x, y, rho, fun = "tOwenT" )};
   \item[row 2:] \code{Phi2xy( x, y, rho, fun = "vOwenT" )};
   \item[row 3:] \code{Phi2xy( x, y, rho, fun = "mOwenT", opt = FALSE )};
   \item[row 4:] \code{Phi2xy( x, y, rho, fun = "mOwenT", opt = TRUE )}.
\end{compactitem}
\end{quote}
The absolute errors, which statistics are collected in Table \ref{Table:2}, were calculated against reference values computed with the accelerated tetrachoric series as described, but with $128$-bit precision. From a user's point of view, only the parameters $x$, $y$ and $\varrho$ needed to be converted into $128$-bit ``mpfr-numbers'' before calling \code{Phi2xy( x, y, rho, fun = "tOwenT" )}, hence \mbox{\code{x <- mpfr( x, precBits = 128 )}} and analogously for $y$ and $\varrho$. In this case, the computation took significantly more time, since all intermediate and final results were $128$-bit ``mpfr-numbers''.
\begin{table}[h]
{\footnotesize
\caption{$\Phi_\varrho^2(x,y)$ computation -- absolute error comparison\label{Table:2}} 
\begin{center}
\begin{tabular}{lrrrrr}
\toprule
\multicolumn{1}{l}{Method}&\multicolumn{1}{c}{Median}&\multicolumn{1}{c}{Mean}&\multicolumn{1}{c}{3rd Qu.}&\multicolumn{1}{c}{99\%}&\multicolumn{1}{c}{Max.}\tabularnewline
\midrule
Tetrachoric series (accelerated)&2,86e-17&3,68e-17&5,37e-17&1,51e-16&2,77e-16\tabularnewline
Vasicek's series (accelerated)&2,57e-17&3,78e-17&5,90e-17&1,65e-16&3,38e-16\tabularnewline
Novel series (atanExt = no)&2,86e-17&3,68e-17&5,38e-17&1,51e-16&2,77e-16\tabularnewline
Novel series (atanExt = yes)&3,36e-17&4,20e-17&6,11e-17&1,56e-16&3,45e-16\tabularnewline
\bottomrule
\end{tabular}
\end{center}}
\end{table}

In this test, $\varrho_{min}\approx-0.99999991$, $\varrho_{max}\approx 0.99999775$, there are $5$ cases with $\varphi_\varrho^2(x,y)>1$ and $\max\varphi_\varrho^2(x,y)\approx 1.38$.

In order to better check the accuracy if $\abs{\varrho}$ is close to $1$, the computation was repeated with the same $x$ and $y$, and $\varrho^*=2\,\Phi(8\varrho)-1$. In this case, more than half of $\abs{\varrho^*}$ are greater than $0.9999$ and none is equal to $1$. The results are in Table \ref{Table:3}.
\begin{table}[h]
{\footnotesize
\caption{$\Phi_{\varrho^*}^2(x,y)$ computation -- absolute error comparison\label{Table:3}} 
\begin{center}
\begin{tabular}{lrrrrr}
\toprule
\multicolumn{1}{l}{Method}&\multicolumn{1}{c}{Median}&\multicolumn{1}{c}{Mean}&\multicolumn{1}{c}{3rd Qu.}&\multicolumn{1}{c}{99\%}&\multicolumn{1}{c}{Max.}\tabularnewline
\midrule
Tetrachoric series (accelerated)&2,92e-17&3,83e-17&5,55e-17&1,60e-16&2,79e-16\tabularnewline
Vasicek's series (accelerated)&1,72e-17&3,38e-17&5,38e-17&1,63e-16&3,01e-16\tabularnewline
Novel series (atanExt = no)&2,92e-17&3,83e-17&5,55e-17&1,60e-16&2,79e-16\tabularnewline
Novel series (atanExt = yes)&3,01e-17&3,91e-17&5,55e-17&1,56e-16&2,95e-16\tabularnewline
\bottomrule
\end{tabular}
\end{center}}
\end{table}
In this test, there are $178$ cases with $\varphi_\varrho^2(x,y)>1$ and $\max\varphi_\varrho^2(x,y)\approx 376$. For both novel series in Table \ref{Table:3}, the maximum absolute error is $1.54\cdot 10^{-14}$ if \eqref{Phirxy6} is not used, improving to $1.21\cdot 10^{-15}$ if \citeauthor{Meyer2013b}'s advice is followed. Similarly applies to the other two series.

Due to the time-consuming $128$-bit precision computation, the presented tests are the only performed tests based on the $128$-bit precision reference values. However, additional tests were made by comparing the results of \code{Phi2xy()} with the results obtained by other functions from packages available on CRAN. For three functions, the accuracy was found to be worse than $2\cdot 10^{-7}$ as measured by the maximum absolute error on the test grid from Subsection \ref{subsec:Results1}. Regarding accuracy, only the function \code{pmvnorm()} with the parameter \code{algorithm = TVPACK} from the \pkg{mvtnorm} package proved to be equivalent to \code{Phi2xy()} and even slightly more accurate with the maximum absolute errors $2.58\cdot 10^{-16}$ and $2.19\cdot 10^{-16}$ on the test sets of triplets $(x,y,\varrho)$ and $(x,y,\varrho^*)$ respectively. The same function with \code{algorithm = GenzBretz} achieved $2.58\cdot 10^{-16}$ and $2.20\cdot 10^{-11}$.

The \code{OwenQ::OwenT()} function also proved to be equivalent to \code{Phi2xy()} as measured by the maximum absolute error on the test sets of triplets, but only when upgraded to compute $\Phi_\varrho^2(x,y)$ in the manner used in \code{Phi2xy()}. This function and \code{Phi2rho::OwenT()} are essentially wrappers for the internal functions \code{tOwenT()}, \code{vOwenT()} and \code{mOwenT()}, which as work\-horses compute series.

In terms of reliability, stability and accuracy of the new methods, no problems or significant differences according to the parameters from the different data sets were detected, except those already described and related to $\abs{\varrho}\approx 1$ and $x-y\,\sgn(\varrho)\simeq 0$. To eliminate them, the equation \eqref{Phirxy6} was derived, which also proved to be successful when using the \code{OwenQ::OwenT()} function, upgraded to compute $\Phi_\varrho^2(x,y)$. Regarding parameters from different data sets, the only detected big difference is in the number of iterations, as can be seen from Figure \ref{Figure:1} in the next subsection.

Two million $\Phi_\varrho^2(x,y)$ computations on Windows 10 and 64-bit \proglang{R} 4.3.1 on the old Intel i7-6500U CPU @ 2.50 GHz and 8 GB RAM, using only one thread of four, lasted $26$, $31$ and $286$ seconds for \code{Phi2xy()}, upgraded \code{OwenQ::OwenT()} and \code{pmvnorm()} with \code{algorithm = TVPACK} respectively. The $128$-bit precision reference values computation took over $17$ hours, compared to $43$ seconds for the double precision one. Based on Tables \ref{Table:1} and \ref{Table:4}, we can conclude that the increased number of iterations has only an insignificant effect on the huge time extension factor.

\FloatBarrier

\subsection{High-precision computation}

All functions in the \code{Phi2rho} package are ready to use the \pkg{Rmpfr} package, which enables using arbitrary precision numbers instead of double precision ones and provides all the high-precision functions needed. It interfaces \proglang{R} to the widely used \proglang{C} Library \code{MPFR} \citep{Fousse2007}. Assuming that the \pkg{Rmpfr} package is loaded, the functions should only be called with the parameters, which are ``mpfr-numbers'' of the same precision. All $128$-bit precision benchmark values used in Subsections \ref{subsec:Results1} and \ref{subsec:Results2} were calculated using \pkg{Rmpfr}.

To get a sense of how the number of iterations depends on the required precision, the test from Subsection \ref{subsec:Results1} was partially repeated using the $128$-bit precision computation. The results are collected in Table \ref{Table:4}. Note that $2^{-128}\approx 2.94\cdot 10^{-39}$ and that in this case the values in the table are not the maximum absolute errors because the comparison values are computed with the same precision.
\setlength{\tabcolsep}{3pt}
\begin{table}[h]
{\footnotesize
\caption{$128$-bit precision computation of $T(h,r)=\Phi_\varrho^2(h,0)-\frac{1}{2}\Phi(h)$ on the test grid\label{Table:4}} 
\begin{center}
\begin{tabular}{lrcrcrr}
\toprule
\multicolumn{1}{l}{\bfseries }&\multicolumn{1}{c}{\bfseries Tetrachoric}&\multicolumn{1}{c}{\bfseries }&\multicolumn{1}{c}{\bfseries Vasicek}&\multicolumn{1}{c}{\bfseries }&\multicolumn{2}{c}{\bfseries Novel series}\tabularnewline
\cmidrule(lr){2-2} \cmidrule(lr){4-4} \cmidrule(lr){6-7}
\multicolumn{1}{l}{}&\multicolumn{1}{c}{(accelerated)}&\multicolumn{1}{c}{}&\multicolumn{1}{c}{(accelerated)}&\multicolumn{1}{c}{}&\multicolumn{1}{c}{(atanExt = no)}&\multicolumn{1}{c}{(atanExt = yes)}\tabularnewline
\midrule
{\bfseries $h$ is a vector}&&&&&&\tabularnewline
~~Max abs. error\phantom{ab}&benchmark&&6.61e-39&&5.51e-39&3.88e-39\tabularnewline
~~Average iter. 1\phantom{ab}&82.8&&75.5&&105.7&55.0\tabularnewline
\midrule
{\bfseries $r$ is a vector}&&&&&&\tabularnewline
~~Max abs. error \phantom{ba}&0.00e+00&&6.61e-39&&5.51e-39&3.88e-39\tabularnewline
~~Average iter. 1 \phantom{ba}&132.9&&120.9&&149.2&72.5\tabularnewline
\midrule
{\bfseries for both}&&&&&&\tabularnewline
~~Average iter. 2 \phantom{ba}&63.2&&57.7&&73.4&40.3\tabularnewline
~~Maximum iter. \phantom{ba}&191&&144&&199&119\tabularnewline
\bottomrule
\end{tabular}
\end{center}}
\end{table}
\setlength{\tabcolsep}{6pt} 

From Tables \ref{Table:1} and \ref{Table:4} can be concluded that there is no significant difference in accuracy between the compared series if the accelerated series are considered, but one of the new series converges significantly faster than the others. It can also be concluded that the Vasicek's series with $\overline{\varrho}$ instead of $\varrho$ is a reasonable alternative to the tetrachoric series \eqref{TetrachA} also for $\varrho^2<\frac{1}{2}$ and remains a reasonable alternative for $\varrho^2>\frac{1}{2}$ if compared to the transformed tetrachoric series \eqref{TetrachB}.

The number of iterations for the computation with the new series with double and $128$-bit precision can be compared in Figure \ref{Figure:1}. Only the first quadrant is presented because others are its mirror images.
\begin{figure}[h]
{\centering
\includegraphics{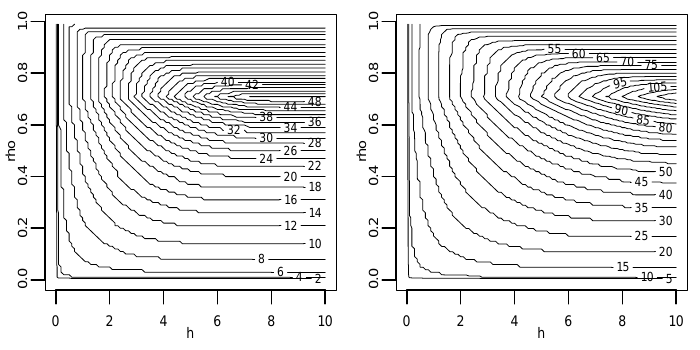}
\caption{$T(h,r)$ computation with the new series (atanExt = yes) with double (left) and $128$-bit (right) precision -- number of iterations\label{Figure:1}}}
\end{figure}

We also provide examples with $r=1$ and $r=-1$ for which reference values can be computed alternatively. $A=\Phi_{\sqrt{2}/2}^2(2.1,0)$ and $B=\Phi_{-\sqrt{2}/2}^2(2.1,0)$ were computed by \eqref{Owen2} and \eqref{Owen3}, and by the right side of \eqref{Phirh0b} by the \code{pnorm()} function from the package \pkg{Rmpfr} as a benchmark. In all cases, the latter was computed by double the precision expressed in bits and used for the former. Both reference values were used to compare how quickly the compared series converge when computing with more than $128$-bit precision. In Tables \ref{Table:5} and \ref{Table:6}, for each of them the absolute error and the number of iterations, which is the same for both computations, is presented. Since $\abs{r}=1$, there is no difference between the original and accelerated tetrachoric series, and the same applies to Vasicek's series.
\begin{table}[h]
{\footnotesize
\caption{A special case of high-precision computing -- absolute error comparison\label{Table:5}} 
\begin{center}
\begin{tabular}{rrcrrrcrrr}
\toprule
\multicolumn{1}{c}{}&\multicolumn{1}{c}{}&\multicolumn{1}{c}{}&\multicolumn{3}{c}{\bfseries Tetrachoric}&\multicolumn{1}{c}{}&\multicolumn{3}{c}{\bfseries Vasicek}\tabularnewline
\cmidrule(lr){4-6} \cmidrule(lr){8-10}
\multicolumn{1}{r}{Bits}&\multicolumn{1}{r}{Precision}&\multicolumn{1}{c}{}& \multicolumn{1}{c}{\phantom{1}A}&\multicolumn{1}{c}{\phantom{1}B}& \multicolumn{1}{r}{\phantom{1}n}&\multicolumn{1}{c}{}&\multicolumn{1}{c}{\phantom{2}A}& \multicolumn{1}{c}{\phantom{2}B}&\multicolumn{1}{r}{\phantom{2}n}\tabularnewline
\midrule
53&1.11e-16&&6.98e-17&8.65e-17&46&&6.98e-17&3.10e-17&47\tabularnewline
64&5.42e-20&&6.89e-21&2.26e-20&57&&6.89e-21&2.26e-20&57\tabularnewline
128&2.94e-39&&5.47e-40&5.89e-41&119&&5.47e-40&5.89e-41&120\tabularnewline
256&8.64e-78&&6.29e-78&3.33e-78&246&&6.29e-78&3.33e-78&246\tabularnewline
512&7.46e-155&&3.65e-156&1.37e-155&500&&4.09e-155&1.37e-155&501\tabularnewline
1024&5.56e-309&&5.04e-309&2.41e-309&1011&&2.25e-309&5.19e-309&1011\tabularnewline
\bottomrule
\end{tabular}
\end{center}}
\end{table}
\begin{table}[h]
{\footnotesize
\caption{A special case of high-precision computing -- absolute error comparison\label{Table:6}} 
\begin{center}
\begin{tabular}{rrcrrrcrrr}
\toprule
\multicolumn{1}{c}{}&\multicolumn{1}{c}{}&\multicolumn{1}{c}{}&\multicolumn{3}{c}{\bfseries Novel series \textnormal{(atanExt = no)}}&\multicolumn{1}{c}{}&\multicolumn{3}{c}{\bfseries Novel series \textnormal{(atanExt = yes)}}\tabularnewline
\cmidrule(lr){4-6} \cmidrule(lr){8-10}
\multicolumn{1}{r}{Bits}&\multicolumn{1}{r}{Precision}&\multicolumn{1}{c}{}& \multicolumn{1}{c}{\phantom{3}A}&\multicolumn{1}{c}{\phantom{3}B}& \multicolumn{1}{r}{\phantom{3}n}&\multicolumn{1}{c}{}&\multicolumn{1}{c}{\phantom{4}A}& \multicolumn{1}{c}{\phantom{4}B}&\multicolumn{1}{r}{\phantom{4}n}\tabularnewline
\midrule
53&1.11e-16&&6.98e-17&3.10e-17&54&&6.98e-17&3.10e-17&22\tabularnewline
64&5.42e-20&&3.40e-20&4.54e-21&65&&6.89e-21&2.26e-20&25\tabularnewline
128&2.94e-39&&5.47e-40&5.89e-41&128&&9.23e-40&1.41e-39&41\tabularnewline
256&8.64e-78&&6.29e-78&3.33e-78&256&&6.29e-78&3.33e-78&69\tabularnewline
512&7.46e-155&&4.09e-155&1.37e-155&511&&4.09e-155&1.37e-155&116\tabularnewline
1024&5.56e-309&&2.25e-309&5.19e-309&1023&&5.04e-309&2.41e-309&199\tabularnewline
\bottomrule
\end{tabular}
\end{center}}
\end{table}
Time-consuming high precision testing was less intensive than double precision one. Since there is no competing package on CRAN ready for high-precision computation of $\Phi_\varrho^2(x,y)$, it could not be supplemented by comparison computations with competing packages.

\section{Conclusion}

If the generally applicable Theorem \ref{Theorem1} is not considered, Theorems \ref{Theorem2} and \ref{Theorem3}, and Corollary \ref{Corollary2} are interesting mainly theoretically, because Theorem \ref{Theorem4}, together with \eqref{Owen2} and \eqref{Owen3}, replaces and supplements them in practice. The series \eqref{Owen6a}--\eqref{Owen6d} enable fast and numerically stable computation, which is often more important than the speed of convergence. From Tables \ref{Table:1} and \ref{Table:4} can be concluded that the computation with \eqref{Owen6a} or \eqref{Owen6c} needs significantly fewer iterations than the computation with the accelerated tetrachoric and Vasicek's series, and from Tables \ref{Table:5} and \ref{Table:6} can be assumed that the advantage increases with increasing precision. For $53$-bit precision (double precision), it required a half, and for $1024$-bit precision only a fifth, of the iterations demanded by the competing series. However, the cost of calculating $\arctan r$ is not taken into account. In connection with this, let us just mention two more questions.

Recalling that $e^{-x}I_n(x)=P(n+1,x)$, the Taylor series of the arctangent function and \eqref{Owen1} imply
\begin{equation*}
   T(h,r)=\tfrac{\arctan r}{2\pi}-\tfrac{1}{2\pi}\sum_{k=0}^\infty\tfrac{(-1)^k r^{2k+1}}{2k+1}\;P(k+1,\tfrac{1}{2}h^2)=\tfrac{1}{2\pi}\sum_{k=0}^\infty\tfrac{(-1)^k r^{2k+1}}{2k+1}\;Q(k+1,\tfrac{1}{2}h^2),
\end{equation*}
which \citet[p.~241]{Fayed2014a} have already noticed. If the above series, \eqref{Eq:Theo2a} and \eqref{Eq:Theo3a} are viewed as functions of the variable $h$, they have a similar structure, but the coefficients, depending on $r$, belong to different arctangent series. Whether the similarity could be explained by a similar Euler transform as the one which transforms the Taylor arctangent series to the Euler's arctangent series seems interesting question, but was not deeply explored.

In Recursion \ref{Recurs1}, the external calculation of the arctangent function is avoided, assuming that the Euler's series \eqref{Euler1} is used for it, and the expectation that there would be no difference in the number of iterations for Recursions \ref{Recurs1} and \ref{Recurs2} if those for calculating the arctangent were also taken into account. However, the use of the external arctangent calculation allows the use of already known and possible future faster converging series. Such is the case with
\begin{equation}\label{Eq:arctan}
   \arctan r=\tfrac{r}{\sqrt{1+r^2}}\sum_{k=0}^\infty\tfrac{(2k-1)!!}{(2k)!!\, (2k+1)}\left(\tfrac{r^2}{1+r^2}\right)^k\quad(r\in\mathbb{R}),
\end{equation}
which is based on the Taylor series of $\arcsin r$ and $\arctan r=\arcsin\frac{r}{\sqrt{1+r^2}}$. Indeed, the number of iterations is not significantly smaller if $\abs{r}\le 1$ and the accuracy obtained by double precision computation is sufficient, and even the square root must be calculated, but it could be significantly smaller in high-precision computation. The Euler's series \eqref{Euler1} is included in \citep{Abramowitz1972} and \citep{Olver2010}, but \eqref{Eq:arctan} is not, even though it converges faster. It can be found in \citep[p.~61, 1.644 1.]{Gradshteyn2015} as
\begin{equation*}
   \arctan r=\tfrac{r}{\sqrt{1+r^2}}\sum_{k=0}^\infty\tfrac{(2k)!}{2^{2k}(k!)^2(2k+1)} \left(\tfrac{r^2}{1+r^2}\right)^k
\end{equation*}
with a reference to the 1922 source. Due to equations $(2k-1)!!=\frac{(2k)!}{2^k k!}$ and $(2k)!!=2^k k!$, the coefficients are the same as those in \eqref{Eq:arctan}. Whether we could also find a corresponding series for the Owen's $T$ function, which would converge faster than \eqref{Owen6b}, remains an open challenge. 



\end{document}